\documentclass[onefignum,onetabnum]{siamonline190516}

\usepackage{amsmath}
\usepackage{amsfonts}
\usepackage{amssymb}
\usepackage{graphicx}
\usepackage{caption,color}
\usepackage{subcaption,enumitem}
\usepackage{textcomp}
\usepackage{algorithmicx}
\usepackage{algpseudocode}
\usepackage{tikz}

\newcommand{\Bignorm}[1]{\Bigl\|#1\Bigr\|}   
\newcommand{\bignorm}[1]{\bigl\lVert#1\bigr\rVert}   

\newcommand{\bR}{\field{R}}        

\newcommand{\ceil}[1]{\ensuremath{\lceil #1 \rceil}}   
\def\Cessum{{\lower 1.5pt \hbox{\large C}}{-}\!\!\!\sum}
\def\CHI{\hbox{\raise .5ex \hbox{$\chi$}}}
\def\esssup{\operatornamewithlimits{ess\,sup}}

\newcommand{\field}[1]{\mathbb{#1}}   

\newcommand{\GLsp}{{\boldsymbol G\kern-0.1em\boldsymbol L}}
\def\Msp{{\boldsymbol M}}   
\def\mv1{\Msp_v^1}   

\newcommand{\norm}[1]{\lVert#1\rVert}   

\def\sabs{{{\raise 0.5pt \hbox{\footnotesize $|$}}}}

\newcommand{\sgn}{\operatorname{sgn}}   
\def\shah{\sqcup \hspace{-0.15em}\sqcup}   
\def\shah1{{\makebox[2.3ex][s]{$\sqcup$\hspace{-0.15em}\hfill $\sqcup$}}}   

\def\sinc{\operatorname{sinc}}   

\def\slb{{{\raise 0.5pt \hbox{\footnotesize $[$}}}}
\def\slcb{{{\raise 0.5pt \hbox{\footnotesize $\{$}}}}
\def\slp{{{\raise 0.5pt \hbox{\footnotesize $($}}}}   

\newcommand{\snorm}{{{\raise 0.5pt \hbox{\footnotesize $\|$}}}}   

\def\srb{{{\raise 0.5pt \hbox{\footnotesize $]$}}}}
\def\srcb{{{\raise 0.5pt \hbox{\footnotesize $\}$}}}}
\def\srp{{{\raise 0.5pt \hbox{\footnotesize $)$}}}}

\def\supp{\operatorname{supp}}   

\def\trace{{\operatorname{trace}}}   

\def\Zdst{{\Zst^d}}   

\def\Zst{{\mathbb Z}}

\graphicspath{{./gplt_scripts/}}

\DeclareGraphicsExtensions{.pdf}

\newcommand{\textapprox}{\raisebox{0.5ex}{\texttildelow}}

\newcommand{\euler}{\mathrm{e}}
\newcommand{\imunit}{\mathrm{i}}

\newcommand{\e}[1]{\ensuremath{\euler^{2\pi \imunit#1}}}

\newcommand{\ip}[2]{\ensuremath{\left<#1,#2\right>}}

\newcommand{\abs}[1]{\ensuremath{\left| #1 \right| }}

\newsiamremark{rem}{Remark}

\DeclareMathOperator{\cov}{Cov}
\DeclareMathOperator{\var}{Var}
\DeclareMathOperator{\mse}{MSE}

\newcommand{\newm}{g}

\newcommand{\Integer}{\mathbb{Z}}

\newcommand{\gridn}{{\mathcal{Q}_N}}
\newcommand{\sset}{{\Omega}}
\newcommand{\ssetgrids}{{\Omega^{\mathrm{grid}}}}
\newcommand{\popr}{\mathcal{X}}
\newcommand{\spd}{S}
\newcommand{\hspd}{\widehat{S}}
\newcommand{\hspdmt}{\widehat{S}^{\mathrm{mt}}}
\newcommand{\hspdr}{\widehat{S}^{\mathrm{pmt}}}
\newcommand{\caro}{{n_\Omega}}
\newcommand{\lambdak}{\lambda_k}
\newcommand{\mat}{T^{\Omega,W}}

\newcommand{\win}{\rho}
\newcommand{\randmat}{G}
\newcommand{\randm}{\widetilde{m}}

\newcommand{\hh}{\mathrm{h}}
\newcommand{\rr}{\mathrm{r}}
\newcommand{\pero}{n_{\partial \Omega}}
\newcommand{\transp}{\mathrm{T}}

\title{Multitaper estimation on arbitrary domains%
\thanks{Both authors contributed equally to this work.
\funding{J. L. R. gratefully acknowledges support from the Austrian Science Fund (FWF): P 29462-N35, from the WWTF grant INSIGHT (MA16-053).}}}

\author{
Joakim And\'en%
\thanks{Department of Mathematics, KTH Royal Institute of Technology, SE-100 44 Stockholm, Sweden and Center for Computational Mathematics, Flatiron Institute, 162 5th Avenue, New York, NY 10010, USA (\email{janden@kth.se}).}%
\and%
Jos\'e Luis Romero%
\thanks{Faculty of Mathematics, University of Vienna, Oskar-Morgenstern-Platz 1, A-1090 Vienna, Austria and Acoustics Research Institute, Austrian Academy of Sciences, Wohllebengasse 12-14, A-1040, Vienna, Austria (\email{jose.luis.romero@univie.ac.at}, \email{jlromero@kfs.oeaw.ac.at}).}%
}

\begin{document}

\maketitle

\begin{abstract}
Multitaper estimators have enjoyed significant success in estimating spectral densities from finite samples using as tapers Slepian functions defined on the acquisition domain.
Unfortunately, the numerical calculation of these Slepian tapers is only tractable for certain symmetric domains, such as rectangles or disks.
In addition, no performance bounds are currently available for the mean squared error of the spectral density estimate.
This situation is inadequate for applications such as cryo-electron microscopy, where noise models must be estimated from irregular domains with small sample sizes.
We show that the multitaper estimator only depends on the linear space spanned by the tapers.
As a result, Slepian tapers may be replaced by proxy tapers spanning the same subspace (validating the common practice of using partially converged solutions to the Slepian eigenproblem as tapers).
These proxies may consequently be calculated using standard numerical algorithms for block diagonalization.
We also prove a set of performance bounds for multitaper estimators on arbitrary domains.
The method is demonstrated on synthetic and experimental datasets from cryo-electron microscopy, where it reduces mean squared error by a factor of two or more compared to traditional methods.
\end{abstract}



\section{Introduction}

Estimating the frequency content of a stochastic process is crucial to many data processing tasks.
For example, in several inverse problems, such as denoising, access to a good noise model is necessary for accurate reconstruction.
With other tasks, such as system identification, this frequency structure is itself the object of interest.

The frequency content of a stationary process $\popr$ over $\Integer^d$ is characterized by its \emph{spectral density} $\spd$, defined as the Fourier series of its autocovariance function \cite{pewa93}.
Spectral estimation is the task of recovering $\spd$ given one or more realizations of $\popr$ over a finite subset $\sset$ of $\Integer^d$.
The challenge in the estimation problem is two-fold:
(i) \emph{stochastic fluctuations} in $\popr$ introduce variance into estimates of $\spd$, and
(ii) \emph{spatial constraints} restrict us to using only samples from the domain $\sset$.
When many realizations are available, the first point may be addressed through ensemble averaging.
For several applications, however, such as geosciences, cosmology, and electron microscopy imaging, only a single realization is available, requiring accurate \emph{single-shot estimators}.

One approach to reducing error in single-shot spectral estimation is to divide the domain $\sset$ into disjoint subsets, compute spectral density estimates for each, and average the result.
In one dimension, with $\sset = \{0, \ldots, N-1\}$ for some $N > 0$, this is typically done by partitioning $\sset$ into $K$ blocks of $N/K$ samples each, computing a spectral estimate, such as a periodogram, for each subset, and averaging.
The $K$ parameter controls the trade-off between the two challenges raised previously: while a higher value of $K$ mitigates the stochastic fluctuations (reducing variance), each periodogram is computed from only $N/K$ samples, reducing their resolution (increasing bias) \cite{be83}.
One problem with this approach, however, is that it introduces artifacts due to the boundaries imposed by the partitioning.
The multitaper estimator, introduced by Thomson \cite{th82}, refines this method by instead computing periodograms over all of $\sset$, but with the samples first multiplied by a set of $K$ sequences called \emph{tapers}.
A favorable trade-off between bias and variance is obtained for specific sequences known as discrete prolate spheroidal sequences, or \emph{Slepian tapers} \cite{slpo61}.

Originally defined for $d = 1$, multitaper estimators naturally generalize to higher dimensions \cite{sl64}.
The Slepian tapers, however, are defined as solutions to an ill-posed eigenvalue problem.
Special geometries, including rectangular grids \cite[Chapter 2]{hola12} and domains involving other symmetries \cite{gr81,hola12,siwado11,sidawi06}, may be analyzed by means of commuting differential operators, yielding well-posed problems that allow for explicit calculation of the tapers.
There are other stable algorithms for calculating Slepian tapers, but these are defined only for specific domains \cite{laho12, lash16, hola17, kash18} or for certain modified tapers \cite{risi95, refs1, refs2, simons2003spatiospectral}.

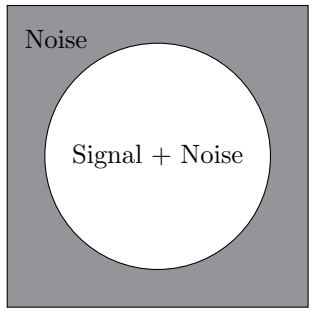
\begin{figure}
\centering
\begin{tikzpicture}
\draw[black, fill=gray] (-2, -2) rectangle (2, 2);
\draw[black, fill=white] (0, 0) circle [radius=1.5cm];

\node at (0, 0) {\small Signal + Noise};
\node at (-1.35, 1.55) {\small Noise};
\end{tikzpicture}
\caption{
\label{fig:domain_ex}
The domain used to estimate the noise spectral density (gray), which excludes the central disk of the image (white).
By excluding the area occupied by the molecule projection, we obtain a more accurate estimate of the noise spectrum.
}
\end{figure}

At times, however, more general acquisition domains are needed, such as the complement of a disk (see Figure \ref{fig:domain_ex}).
This geometry arises in cryo-electron microscopy (cryo-EM), where noisy projection images contain signal on a central disk and clean samples of the noise process are found only outside of that disk.
In geosciences, a similar problem occurs when a physical quantity needs to be sampled on a subregion $\sset$ of the earth \cite{simons2000isostatic, sidawi06, simons2006spherical, siwado11, plsi14}.
For these general domains $\sset$, all available methods inherit the instability of the underlying eigenvalue problem \cite{siwado11}.
That being said, remarkable results have been obtained by direct application of standard eigenvalue solvers and using the resulting eigenvectors as tapers, despite these not being true Slepian tapers \cite{plsi14,sidawi06,siwado11,hasi12}.
The behavior and performance of these pseudo-multitaper estimators, however, is not well understood.

An alternative solution is to partition this irregular domain into smaller rectangular regions and apply the tensor Slepian multitaper estimator to those.
Since each subdomain is a rectangle, Slepian functions are readily calculated on these domains.
This strategy, however, lacks the simplicity of Thomson's multitaper estimator, and may typically only be implemented suboptimally.

In this work, we show that the multitaper estimator only depends on the linear span of the tapers used.
Consequently, we may calculate it with any set of tapers that span the same subspace as the Slepian tapers.
This explains the success of standard eigendecomposition algorithms applied to the ill-posed eigenvalue problem since a typical failure mode results in vectors with the same span as the true eigenvectors.
To ensure this fortuitous behavior, we propose replacing the standard eigendecomposition with a block eigendecomposition, extracting a basis for the subspace spanned by the leading $K$ eigenvectors.
This problem is well-posed, because of the large spectral gap between the first $K$ eigenvalues and the rest of the spectrum, which we validate.

Furthermore, we provide a mean squared error bound for the multitaper estimator on arbitrary domains.
This generalizes previous performance bounds for the one-dimensional case \cite{abro17}.
We validate these bounds numerically, showing how they correctly predict the error decay as a function of the acquisition geometry.

The proposed method is evaluated on synthetic data, where it is shown to perform comparably to the Slepian multitaper method for rectangular domains.
However, we also show that the method performs equally well on the complement of a disk.
We also evaluate the method on cryo-EM images, both synthesized and from experimental datasets, where we also achieve good performance compared to previous approaches.

Section \ref{sec_problem} introduces the spectral estimation problem for general domains.
The multitaper spectral estimator is introduced in Section \ref{sec_mt}.
Section \ref{sec_analysis} provides a bound for the mean squared error of the multitaper estimator for arbitrary domains.
We show that the estimator only depends on the linear span of the tapers in Section \ref{sec_rt} and use this to provide an implementation using proxy tapers.
Finally, Section \ref{sec_nums} illustrates the performance of the proposed estimator using numerical experiments.%
\footnote{Python code for reproducing the results of this work may be found at \url{https://github.com/janden/pmte}}

\paragraph{Notation}
In the following, $:=$ specifies the definition of a quantity.
For a vector $x \in \mathbb{R}^d$, we let $\abs{x} := \big(\sum_{k=1}^d \abs{x_k}^2\big)^{1/2}$ be its Euclidean norm.
The $C^2$-norm of a twice continuously differentiable function $f: \mathbb{R}^d \to \mathbb{C}$ is
\begin{align*}
\norm{f}_{C^2} := \max \left\{\sup_{x\in\mathbb{R}^d}\abs{f(x)}, \sup_{x\in\mathbb{R}^d}\abs{\partial_{x_j} f(x)},
\sup_{x\in\mathbb{R}^d}\abs{\partial_{x_j}\partial_{x_{j'}} f(x)}: j,j'=1,\ldots,d \right\},
\end{align*}
whenever finite.
We say that $f$ is \emph{1-periodic} if $f(x+k)=f(x)$ for all $k \in \mathbb{Z}^d$. The indicator function $1_\sset: S \to \mathbb{R}$ of a subset $\sset \subset S$ is defined as
\begin{equation}
1_\sset(x) := \left\{ \begin{array}{ll} 1, & \mbox{if}~x\in\sset \\ 0, & \mbox{otherwise.} \end{array} \right.
\end{equation}
The inner product between two vectors $q$ and $q'$ in $\mathbb{R}^d$ is given by $\ip{u}{v'} = \sum_{j=1}^d u_j v'_j$.
Finally, $\mathbb{E}(X)$ denotes the expected value of the random variable $X$ and $\ceil{x}$ denotes the smallest integer $k$ such that $k \ge x$.

\section{Spectral estimation on irregular domains}
\label{sec_problem}
\mbox{}
Let us consider a real-valued Gaussian, zero-mean, stationary process $\popr$ defined on an infinite grid $\mathbb{Z}^d$.
Our goal is to estimate the covariance matrix
\begin{align*}
\cov[q,q'] = \mathbb{E} \left\{ \popr[q] \popr[q'] \right\} \qquad q,q' \in \Zdst.
\end{align*}
Since $\popr$ is stationary, $\cov[q,q']$ only depends on the difference $q-q'$.
We therefore rewrite the matrix in terms of the autocovariance function $r$, giving
\begin{align*}
\cov[q,q'] = r[q-q'] \qquad q,q' \in \Zdst.
\end{align*}
The covariance information is equivalently encoded in the \emph{spectral density}
\begin{align*}
\spd(\xi) = \sum_{q \in \Zdst} r[q] \e{\ip{q}{\xi}}, \qquad \xi \in [-1/2,1/2]^d.
\end{align*}
Given one or several realizations of $\popr$ on a subset $\sset \subseteq \mathbb{Z}^d$ of cardinality $\caro$, we wish to estimate $\spd$.
This is known as the \emph{spectral estimation problem} \cite{pewa93}.

\begin{figure}[t]
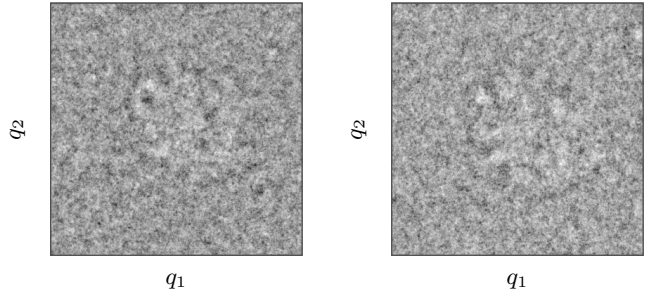

\centering
\begin{subfigure}{.25\textwidth}
\centering
\input{gplt_scripts/cryo_exp_sig_noise1_gplt}
\end{subfigure}
\begin{subfigure}{.25\textwidth}
\centering
\input{gplt_scripts/cryo_exp_sig_noise2_gplt}
\end{subfigure}
\caption{%
\label{fig:cryo_ex}
Two images from the EMPIAR-10028 cryo-electron microscopy dataset \cite{rib80s} as functions of the spatial position $(q_1, q_2)$.
These are tomographic projections of a ribosome, frozen in a thin sheet of vitreous ice and imaged using a transmission electron microscope.
The low electron dose results in very noisy images (the actual ribosome is in the center of each image).
}
\end{figure}

An instance of this problem arises in imaging.
Suppose that an image is defined on a grid \[\gridn=\{0,\ldots,N-1\}^2\] of pixels, and $\popr[q]$ represents additive noise at the pixel $q \in \gridn$.
A good estimate of the noise distribution is then needed for denoising and other tasks.
For example, in single-particle cryo-EM, molecules are imaged by freezing them in a thin layer of ice and recording their tomographic projections using an electron microscope \cite{fr06}.
To reduce specimen damage, electron dose is kept low, resulting in exceptionally noisy images, as shown in Figure \ref{fig:cryo_ex}.
The projection of the molecule is expected to lie in the center of the image, so pure noise samples are only found outside of a central disk.
A reasonable domain for noise estimation is therefore the complement of that disk
\begin{align}
\label{eq_sset_def}
\sset := \left\{q \in \gridn : \abs{q-(N/2,N/2)} > R \right\},
\end{align}
as illustrated in Figure \ref{fig:domain_ex}.

Another instance is found in geosciences, where $\popr$ is a physical quantity on an approximately flat portion of the earth, and measurements are only available on a subregion $\sset$, corresponding, for example, to a continent \cite{siwado11}.
A more sophisticated model replaces the Cartesian grid for a grid on the sphere \cite{sidawi06,plsi14}.

For both of these applications, we only have access to a single realization of the process whose spectral density we wish to estimate.
In cryo-EM, this is due to changing experimental conditions between projection images, such as non-stationary optical parameters or variation in ice thickness, while in geosciences, only one realization (i.e., one planet) exists.
We therefore require a \emph{single-shot estimator} of \spd.
Instead of averaging estimates from several realizations of $\popr$, a single-shot estimator only takes a single realization as input.
The total number of samples is therefore $\caro$.
This low sample size often results in high variance, so we need a method which regularizes this estimate.
The aim of the present work is to study the performance of one such method, the multitaper estimator, and provide an explicit implementation.

\section{Multitaper estimators}
\label{sec_mt}

The inherent variability of $\popr$ induces error into any estimate of its spectral density $\spd$.
With access to only a single realization, ensemble averaging cannot be used to reduce the error.
Instead, we must impose some constraint on the estimate.
This often involves post-processing of the spectral density estimate by smoothing or fitting parametric models \cite{brda91,pewa93}.

The multitaper spectral estimator provides another solution to this challenge \cite{th82}.
Initially introduced for one-dimensional signals, we present it here for signals of arbitrary dimension.
Let $m \in \ell^2(\mathbb{Z}^d)$ be a function of unit $\ell^2$-norm supported on $\sset$, that is, satisfying
\begin{align}
m[q]=0, \qquad \mbox{if }q \notin \sset.
\end{align}
We now define the \emph{tapered periodogram} of $\popr$ with taper $m$ as
\begin{align}
\label{eq_mt}
\hspd_{m}(\xi) := \abs{\sum_{q \in \sset} m[q] \popr[q] \e{\ip{q}{\xi}}}^2, \qquad \xi \in [-1/2,1/2]^d.
\end{align}
Given an orthonormal family $m_0, \ldots, m_{K-1} \in \ell^2(\mathbb{Z}^d)$ of tapers supported on $\sset$, we define the corresponding multitaper estimator as the average
\begin{align}
\hspdmt(\xi) := \frac{1}{K}\sum_{k=0}^{K-1} \hspd_{m_k}(\xi) \qquad \xi \in [-1/2,1/2]^d.
\end{align}
The error of $\hspdmt$ in estimating $\spd$ depends on the choice of tapers $m_0, \ldots, m_{K-1}$.

The tapers of the traditional multitaper estimator are defined using the following \emph{spectral concentration problem}:
\begin{equation}
\label{eq_spec_problem}
\begin{aligned}
&\mbox{maximize } \int_{[-W/2,W/2]^d}
\abs{\sum_{q \in \sset} m[q] \e{\ip{q}{\xi}}}^2 d\xi,
\\
&\qquad \mbox{subject to:} \sum_{q \in \sset} \abs{m[q]}^2 = 1,
\mbox{ and }
\supp(m) \subseteq \sset,
\end{aligned}
\end{equation}
where
\begin{equation}
\label{eq_W}
W = (K/\caro)^{1/d},
\end{equation}
and $\caro$ is the number of elements in $\sset$.
The parameter $W$ is known as the \emph{bandwidth} of the estimator.
It is chosen so that
\begin{align}
\label{eq_KW}
K=\ceil{\caro \times W^d},
\end{align}
and controls the amount of smoothness imposed on the estimate $\hspdmt$.
More concretely, the spectral concentration problem \eqref{eq_spec_problem} for $d = 2$ reduces to
\begin{equation}
\begin{aligned}
&\mbox{maximize } \int_{[-W/2,W/2]^2}
\abs{\sum_{(q_1, q_2) \in \sset} m[q_1, q_2] \e{(q_1\xi_1 + q_2\xi_2)}}^2 d\xi_1 d\xi_2,
\\
&\qquad \mbox{subject to:} \sum_{(q_1, q_2) \in \sset} \abs{m[q_1,q_2]}^2 = 1,
\mbox{ and }
\supp(m) \subseteq \sset,
\end{aligned}
\end{equation}
with $W = \sqrt{K/\caro}$.

The Slepian tapers $m_0, \ldots, m_{K-1}$ are defined as the set of mutually orthogonal solutions to \eqref{eq_spec_problem}.
Specifically, the first taper $m_0$ is the solution to \eqref{eq_spec_problem}, while $m_1$ is the solution to \eqref{eq_spec_problem} with the constraint that $m_1 \perp m_0$, and so on.
In general, $m_j$ is the solution to \eqref{eq_spec_problem} subject to $m_j \perp m_0, \ldots, m_j \perp m_{j-1}$.

Alternatively, \eqref{eq_spec_problem} may be formulated as the maximization of the quadratic form corresponding to the truncated $d$-Toeplitz matrix
\begin{align}
\label{eq_mat}
\mat[q,q'] = \left\{ \begin{array}{ll} W^d \,\sinc \big(W(q-q')\big), & \mbox{if}~q, q' \in \sset, \\ 0, & \mbox{otherwise}, \end{array} \right.
\end{align}
where $q, q' \in \mathbb{Z}^d$ and $\sinc$ is the $d$-dimensional normalized sinc function
\begin{align}
\label{eq_sinc}
\sinc_d(u) := \prod_{k=1}^d \frac{\sin(\pi u_k)}{\pi u_k}, \qquad u = (u_1, \ldots, u_d) \in \mathbb{Z}^d.
\end{align}
The matrix $\mat$ is indexed by vectors $q, q' \in \mathbb{Z}^d$, which means that it maps the space of $d$-dimensional arrays to itself.
This is done by identifying a $d$-dimensional array with a one-dimensional array through lexicographical ordering of indices.
The Slepian tapers $m_0, \ldots, m_{K-1}$ are thus the top $K$ eigenvectors of $\mat$, satisfying
\begin{align}
\label{eq_eig}
\mat m_k = \lambdak m_k, \qquad k \in \{0, \ldots, K-1\},
\end{align}
where $\lambda_0 \ge \lambda_2 \ge \ldots \ge \lambda_{K-1}$.
For $d = 2$, we may substitute \eqref{eq_mat} and \eqref{eq_sinc} into \eqref{eq_eig}, obtaining
\begin{align}
\sum_{(q'_1, q'_2) \in \sset} \frac{\sin(\pi W (q_1 - q'_1)) \sin(\pi W (q_2 - q'_2))}{\pi^2(q_1 - q'_1)(q_2-q'_2)} m_k[q'_1, q'_2] = \lambdak m_k[q_1, q_2], \qquad (q_1, q_2) \in \sset,
\end{align}
and $m_k[q_1, q_2] = 0$ for $(q_1, q_2) \notin \sset$.

\begin{figure}[t]
\centering
\input{gplt_scripts/spectrum_gplt}
\caption{Eigenvalues of $\mat$ for $d = 1$, $K = 7$, and $\sset = \{1, \ldots, 32\}$.}
\label{fig:eig}
\end{figure}

Unfortunately, serious numerical difficulties arise when attempting solve \eqref{eq_eig} numerically.
Indeed, the first \textapprox $K$ eigenvalues form a \emph{plateau profile}, all clustering around $1$ (see Figure \ref{fig:eig}), resulting in small spectral gaps between successive eigenvalues.
For increasing $K$, small spectral gaps result in an ill-posed eigenvector problem \cite{st73,golub-vanloan}, making direct calculation of $m_0, \ldots, m_{K-1}$ in finite precision very challenging.

Under certain circumstances, $\mat$ may be replaced with a commuting differential operator whose spectrum does not exhibit the sample plateau.
This is the case in one dimension when $\sset = \{0, \ldots, N-1\}$, yielding a well-posed eigenvector problem for the Slepian sequences \cite[Chapter 2]{hola12}.
For $d > 1$, we may similarly take $\sset$ as a subgrid of $\Zdst$, that is, $\sset = \{0, \ldots, N-1\}^d$.
In this case, Slepian tapers are tensor products of one-dimensional Slepian sequences \cite{ha17}.  

The existence of a commuting differential operator was referred to as a ``lucky accident'' by Slepian \cite{sl83}.
Similar devices in two dimensions are only available for special cases involving radial symmetries \cite{gr81,hola12,siwado11} or for polar caps in spherical geometry \cite{sidawi06}.
For non-symmetric domains, there are no adequate commuting differential operators \cite{MR847527, MR627121, MR745089}.
Other stable numerical strategies exist only for particular domains \cite{laho12,lash16, hola17, kash18} or for modified tapers that have analytic expressions \cite{risi95}.
For more general domains \emph{no useful symmetries seem to be available}.
Consequently, the calculation of the Slepian tapers for multitaper estimators on irregular domains is affected by numerical instability.
As a consequence of the analysis in this work, however, such instabilities do not preclude a stable implementation of the multitaper estimator.
Indeed, even if the calculation of the Slepian tapers is severely ill-posed, effective numerical proxies are available (see Section \ref{sec_rt}).

Multitaper estimators on irregular domains have been studied previously, notably by Bronez, who referred to the problem of ``irregularly sampling of multidimensional processes'' (referring to the geometry of the set of available samples) and named the corresponding tapers ``generalized prolate spheroidal sequences'' \cite{bronez1988spectral}.
More recently, multitaper estimators associated with irregular domains, including the more challenging setting of spherical geometries, have been instrumental in geosciences and climate analysis \cite{hola12,hasi12}.
Irregular spectra are also relevant in the one-dimensional setting, such as in the field of cognitive radio \cite{ha05}, where the opportunistic occupation of transmission frequencies leads to complex geometries that can be leveraged through carefully designed irregular sampling patterns \cite{coel14}.
In the absence of a commuting operator, practitioners often calculate the tapers by a direct eigenvalue decomposition of $\mat$, an operation that is admittedly unstable but effective in practice \cite{plsi14,sidawi06,siwado11,hasi12}.
A potential reason is that, even when the eigendecomposition fails, it may still yield vectors with the same span as the eigenvectors, which, as shown below in Proposition \ref{prop_span}, perform the same for multitaper estimation (see Section \ref{sec_rt}).

\section{Analysis of multitaper estimator}
\label{sec_analysis}

Let us now consider the performance of the multitaper estimator $\hspdmt$.
Previous analysis for the one-dimensional case \cite{abro17} relies on the following aggregated measure for the spectral resolution of the tapers, known as the \emph{accumulated spectral window}:
\begin{align}
\label{eq_thom_sw}
\win(\xi) := \frac{1}{K} \sum_{k=0}^{K-1} \abs{\sum_{q \in \sset} m_k[q] \e{\ip{q}{\xi}}}^2 \qquad \xi \in [-1/2,1/2]^d.
\end{align}
This window determines the smoothness imposed on $\hspdmt$ and therefore controls the bias, variance, and consequently mean squared error (MSE) of the estimator.
Indeed, for a multitaper estimator based on Slepian tapers, we have the following estimate, which is a minor extension of a result from Lii and Rosenblatt \cite{liro08} (the proof is provided in Appendix \ref{sec_estimates}).
\begin{proposition}
\label{prop_mt_general}
Suppose that the spectral density $\spd$ of $\popr$ is a $1$-periodic bounded function.
If $\win$ is defined from the Slepian tapers by \eqref{eq_thom_sw}, the multitaper estimator $\hspdmt(\xi)$ satisfies
\begin{align*}
&\var \left\{ \hspdmt(\xi) \right\} := \mathbb{E} \left\{ \abs{ \hspdmt(\xi) - \mathbb{E}\left\{\hspdmt(\xi)\right\} }^2 \right\} \le \frac{C}{K} \sup_{\xi \in [-1/2,1/2]^d} S(\xi)^2
\end{align*}
for some constant $C \ge 0$, and
\begin{align*}
&\mathrm{Bias} \left\{\hspdmt(\xi) \right\} := \mathbb{E} \left\{ \abs{\hspdmt(\xi)-\spd(\xi)} \right\} = \abs{\spd(\xi)-\spd*\win(\xi)}.
\end{align*}
\end{proposition}
Note that the above result can also be proved for an arbitrary set of tapers.
In \cite{abro17} it was shown that the spectral window $\win$ associated with Thomson's classical (that is, one-dimensional) multitaper estimator resembles a bump function localized in the interval $[-W/2,W/2]$ and provided concrete error estimates \cite{abro17}.
Such description of the spectral window, combined with Proposition \ref{prop_mt_general}, leads to concrete MSE bounds for the classical multitaper as a function of $K$, thus elaborating on the more qualitative analysis by Lii and Rosenblatt \cite{liro08}.
These types of estimates are also instrumental in the analysis of multitapering for slowly evolving spectral densities \cite{xidi18}.

As a first contribution, we extend the description of the spectral window $\rho$ to arbitrary dimension and
general acquisition domains. We let $\pero$ denote the \emph{digital perimeter} of 
$\sset$:
\begin{align*}
\pero = \sum_{q \in \Zdst} \sum_{j=1}^d \abs{1_\sset(q+e_j) - 1_\sset(q)},
\end{align*}
where $\{e_j:j=1,\ldots,d\}$ is the canonical basis of $\Zdst$.
In the following, we will assume that $W^{d-1} \pero \geq 1$ to avoid degenerate cases.

\begin{theorem}
\label{th_specwin}
Consider the spectral window $\win$ defined by the Slepian tapers $m_0, \ldots, m_{K-1}$ obtained from \eqref{eq_spec_problem}.
Assume that $W^{d-1} \pero \geq 1$, where $W$ and $K$ are related by \eqref{eq_W}.
Then
\begin{align}
\label{eq_con}
\int_{[-1/2,1/2]^d} \abs{\win(\xi) - \frac{1}{W^d} 1_{[-W/2,W/2]^d}(\xi)} d\xi
 \le C \frac{\pero W^{d-1}}{K} \left[ 1+ \log\left( \frac{\caro}{\pero} \right) \right],
\end{align}
for some constant $C \ge 0$.
\end{theorem}
Related results in the context of the short-time Fourier transform can be found in \cite{abgrro16,abpero17}.
The proof of Theorem \ref{th_specwin} is postponed to Appendix \ref{sec_estimates}.

Combining Proposition \ref{prop_mt_general} and Theorem \ref{th_specwin}, we obtain
MSE bounds for the multitaper estimator on general domains $\sset$.
We also present a simplified expression for the bound, valid in the so-called {\em fine-scale regime}.
Here, we consider a class of acquisition domains $\sset$ for which there is a constant $C>0$ such that $\pero \le C \caro^{\frac{d-1}{d}}$.
An instance of this regime occurs, for example, if $\sset$ arises from increasingly fine-scale discretizations of a certain continuous subset of $\mathbb{R}^d$, where $C$ depends on the smoothness of the subset.

\begin{theorem}
\label{th_mse}
Suppose that the spectral density $\spd$ of $\popr$ is a $1$-periodic $C^2$ function.
Assume that $K \pero^{d/(d-1)} \geq \caro$.
Then there exists a constant $C \ge 0$, such that the multitaper estimator $\hspdmt$ with $K$ tapers satisfies the mean squared error bound
\begin{align}
\label{eq_mse}
\mse \left\{\hspdmt (\xi)\right\} &:= \mathbb{E} \left\{ \abs{\hspdmt(\xi) - \spd(\xi)}^2 \right\} \\
& \le C \norm{\spd}^2_{C^2}
\left(
\frac{K^{4/d}}{\caro^{4/d}} + \frac{\pero^2 }{\caro^{2-2/d}K^{2/d}} \left[ 1+ \log\left( \frac{\caro}{\pero} 
\right) \right]^2 + \frac{1}{K}
\right).
\end{align}
In particular, when $d=2$, and, in the fine-scale regime (where $\pero$ is of the order $\sqrt{\caro}$), the choice 
$K = \lceil \caro^{2/3} \rceil$ (or, equivalently, $W = \caro^{-1/6}$) gives
\begin{align}
\label{eq_mse_simple}
\mse \left\{\hspdmt (\xi)\right\} \le C \caro^{-2/3} \cdot \log^2\left( \caro \right) \cdot 
\norm{\spd}^2_{C^2}.
\end{align}
\end{theorem}
The new constant in
\eqref{eq_mse_simple} depends on the constant comparing $\pero$ and $\sqrt{\caro}$.
The proof of Theorem \ref{th_mse} is postponed to Appendix \ref{sec_estimates}.

\begin{rem}
For $d=1$, the choice $K = \lceil \caro^{4/5} \rceil$ in \eqref{eq_mse}
(or, equivalently, $W = \caro^{-1/5}$) gives
\begin{align}
\label{eq_mse_simple_1}
\mse \left\{\hspdmt (\xi)\right\} \le C \caro^{-4/5} \norm{\spd}^2_{C^2},
\end{align}
recovering the result in \cite{abro17}.
In this setting, the minimax risk corresponding to the stronger error measure given by the expected operator 
norm of the covariance matrix satisfies
\begin{align}
\label{eq_smm}
A n^{-4/5} \log(n)^{4/5}
\le \inf_{\widehat{\spd}} \sup_{\spd}
\mathbb{E} \left\{ \sup_\xi \left|\widehat{\spd}(\xi) - \spd(\xi)\right|^2 \right\} \le B n^{-4/5} \log(n)^{4/5},
\end{align}
for some $A, B > 0$.
Here, the infimum is taken among all estimators $\widehat{\spd}$ based on $n$ consecutive samples, and the first $\sup$ is over all spectral densities $\spd$ with H\"{o}lder exponent $2$ satisfying a certain smoothness bound, which determines the constants $A,B$ \cite{MR681465,MR3055254}; see also \cite{kabanava2017masked}.
We are unaware of benchmarks for the spectral estimation problem related to two-dimensional acquisition domains. (Expressions for the MSE depending on the signal and noise power are however available, see, e.g. \cite{simons2006spherical, dahlen2008spectral, plattner2014potential}.)
\end{rem}

\begin{rem}
Theorem \ref{th_mse} follows by combining Proposition \ref{prop_mt_general} and Theorem \ref{th_specwin}
and thus provides individual estimates for the bias and variance of the multitaper estimator. A small variation of the proofs yields similar results for \emph{eigenvalue weighted estimators} -- see \cite[Theorem 2.2]{abro17}.
\end{rem}

\begin{rem}
The estimates leading to Theorem \ref{th_specwin} are also relevant in numerical analysis. For 
example Proposition \ref{prop_tr} below improves on \cite{mahu17}.
\end{rem}

\begin{rem}
The above results may be generalized to cover arbitrary frequency profiles for the tapers instead of squares 
$[-W/2, W/2]^d$.
For example, in applications where the spectral density does not display any particular anisotropy related to 
the axes, a disk might be a better choice. For simplicity we do not pursue such generalizations in the present work. 
\end{rem}

\begin{rem}[Shannon number]
Theorem \ref{th_mse} and the technical lemmas in Appendix \ref{sec_estimates}
provide non-asymptotic bounds for certain heuristic calculations concerning the so-called \emph{Shannon 
number} \cite{siwado11}. These involve the sum of the most significant eigenvalues of the spectral 
concentration problem for irregular domains, and are typically formulated in the large-scale asymptotic 
regime (see also \cite[Section 7]{sidawi06}).
\end{rem}

\section{Proxy Slepian tapers}
\label{sec_rt}

As discussed in Section \ref{sec_mt}, calculating the Slepian tapers for arbitrary domains is a difficult task due to the inherent instability of the underlying eigenproblem.
Indeed, the condition number for the calculation of a single eigenvector is inversely proportional to its distance to the rest of the spectrum---see, for example, \cite[Equation 3.45]{MR3396212} and \cite[Section 2.5]{MR3325822}---and this spectral gap is small because of the plateau spectral profile of $\mat$ for large $K$.
However, practitioners resort to such direct methods with remarkable results \cite{plsi14,sidawi06,siwado11,hasi12}.
One possible explanation is that, unless $K$ is small enough, a standard finite-precision eigenvalue routine may fail to compute the true Slepian tapers, but will still compute an orthonormal basis for their linear span.
(See also \cite{plattner2014potential} for examples where standard eigenvalue routines succeed or fail, and a discussion of possible workarounds.)
As we shall see below, however, these tapers still yield the desired multitaper estimator, since the latter only depends on the tapers through their span.
\begin{proposition}
\label{prop_span}
Let $\{\newm_0, \ldots, \newm_{L-1}\}$ and $\{\widetilde \newm_0, \ldots, \widetilde \newm_{L-1}\}$ be two orthonormal sets that span the same linear space $V \subseteq \ell^2(\mathbb{Z}^d)$.
Then the corresponding multitaper estimators coincide, that is,
\begin{align}
\frac{1}{L}\sum_{k=0}^{L-1} \hspd_{\newm_k}(\xi) = \frac{1}{L}\sum_{k=0}^{L-1} \hspd_{\widetilde \newm_k}(\xi) \qquad \xi \in [-1/2,1/2]^d.
\end{align}
\end{proposition}
The proof is found in Appendix \ref{sec_estimates}.

While this explains the partial success of standard eigenvalue routines when applied to the ill-posed Slepian eigenproblem, the above result also indicates a more straightforward approach.
Instead of trying to solve the standard diagonalization of $\mat$---which may fail but still give usable tapers---we perform a block diagonalization.
In other words, we calculate a basis for the span of the top $K$ eigenvectors of $\mat$.
Indeed, there is a large spectral gap between the first $K$ eigenvalues the rest of the spectrum (which Theorem \ref{th_specwin} and the estimates in Appendix \ref{sec_estimates} validate to some extent).
As a result, computing the associated subspace is a well-posed problem \cite{st73,golub-vanloan}.

We propose to compute these \emph{proxy Slepian tapers} by applying a block power method to $\mat$.
These are then used to compute the multitaper spectral estimator for a given realization of a stationary process.
The resulting algorithm is presented as Algorithm \ref{algo:rt}.

\begin{algorithm}
\begin{algorithmic}
\Function{Estimate}{$\popr[q]$ for $q \in \sset$, $K$, $T$, $\xi$}
\State Let $W \gets \ceil{\caro \times W^d}$.
\State Draw a matrix $\randmat \in \mathbb{R}^{\sset \times K}$ with i.i.d. $\mathcal{N}(0, 1)$ elements.
\For{$t \gets 1, T$}
    \State Compute QR decomposition: $Q \cdot R = \mat \cdot \randmat$.
    \State Set $\randmat \gets Q$.
\EndFor
\State Let $\randm_0, \ldots, \randm_{K-1}$ be the columns of $\randmat$ and
\begin{align*}
\hspdr(\xi) := \frac{1}{K}\sum_{k=0}^{K-1} \hspd_{\randm_k}(\xi).
\end{align*}
\State Return $\hspdr(\xi)$.
\EndFunction
\end{algorithmic}
\caption{%
\label{algo:rt}
The proxy Slepian multitaper estimator.
}
\end{algorithm}

In exact arithmetic, the column space of $\randmat$ converges to the span of the Slepian tapers $\{m_0, \ldots, m_{K-1}\}$ as we increase $T$.
Consequently, by Proposition \ref{prop_span}, $\hspdr(\xi)$ converges to $\hspdmt(\xi)$.
Since the gap between the $K$th and the $(K+1)$th eigenvalues of $\mat$ is typically non-negligible \cite{zhu2018eigenvalue, zhu2017roast}, subspace convergence is relatively fast,
when errors are measured by the maximal singular value of the difference between the orthogonal projections onto the true and calculated subspaces (operator norm) \cite[Theorem 8.2.2]{golub-vanloan}.
Although convergence of the column space of $\randmat$ 
in operator norm is sufficient for accurate reproduction of the Slepian multitaper estimate, it is not necessary.
Indeed, the estimates in Section \ref{sec_estimates} indicate that the multitaper estimator is robust under small errors in the calculation of the span of the Slepian tapers, when measured by the \emph{average} of the singular values of the difference between the orthogonal projections onto the true and calculated subspaces (normalized trace norm). A moderate number of iterations $T$ may be sufficient to control such average and thus accurately calculate the multitaper estimator. Indeed, a single iteration $T = 1$ is often sufficient for many applications. For special domains, such as rectangles, precise non-asymptotic estimates on the spectrum of $\mat$ allow one to quantify the previous remarks, as in \cite{zhu2018eigenvalue, zhu2017roast, karnikbandlimited}.
We expect that a similar analysis of the complexity of Algorithm \ref{algo:rt} for general domains is also possible.

We also note that the only step in Algorithm \ref{algo:rt} that depends on the data is the last one.
As a result, we may precompute the proxy tapers $\randm_0, \ldots, \randm_{K-1}$ once and apply them to the data as many times as necessary.

\begin{rem}
The matrix $\mat$ is supported on $\sset \times \sset$; in Algorithm \ref{algo:rt} it is treated as an element of $\bR^{\sset \times \sset}$.
\end{rem}

\begin{rem}
Applying $\mat$ to a vector involves an insertion followed by a convolution and a truncation.
Insertion and truncation are diagonal operators, while the convolution is a Toeplitz operator that can be applied using a fast Fourier transform.
Multiplication of a vector by $\mat$ is thus achieved in $\mathcal{O}(\caro \log \caro)$ time.
\end{rem}

\begin{rem}
Several authors have proposed numerical recipes to produce tapers adapted to irregular acquisition domains, resorting for example to QR decompositions and Karhunen--Lo\`{e}ve expansions;  see, e.g.,
\cite{refs1,refs2,refs3}. We are unaware of corresponding performance results.
\end{rem}

\section{Numerical results}
\label{sec_nums}

To empirically evaluate the performance of the proposed proxy multitaper estimator, we perform a few numerical experiments on synthetic data.
First, we numerically validate the theoretical results of Theorems \ref{th_specwin} and \ref{th_mse}.
We then compare the proxy tapers to standard Slepian tapers on rectangular domains, where they are shown to perform similarly.
Since the proxy tapers are easily computed on arbitrary domains, we also demonstrate its behavior on the complement of a disk.
We also evaluate their mean squared error on single-particle cryo-EM data, for both synthetic and experimental images.

\subsection{Empirical error analysis}

\begin{figure}[t]
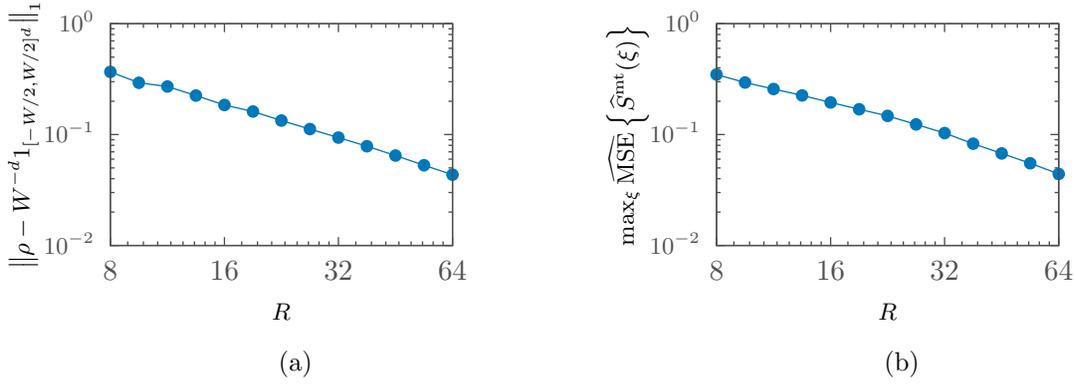

\centering
\begin{subfigure}{0.45\textwidth}
\input{gplt_scripts/rho1_single_gplt}
\caption{}
\label{fig:specwin_validation}
\end{subfigure}
\begin{subfigure}{0.45\textwidth}
\input{gplt_scripts/mse2_single_gplt}
\caption{}
\label{fig:mse_validation}
\end{subfigure}
\caption{
The evolution of spectral window error and spectral estimation error for different disks of radius $R$ within grids $\gridn$ of size $N = 128$ and $d = 2$.
(a) Spectral window error $\int_{[-1/2,1/2]^d} |\win(\xi) - W^{-d} 1_{[-W/2,W/2]^d}(\xi)| d\xi$ as a function of the radius $R$.
The bandwidth $W$ is $1/8$.
(b) The maximum estimated mean squared error $\widehat{\mse} \left\{ \hspdmt(\xi) \right\}$ over $\xi \in [-1/2,1/2]^d$ as a function of $R$.
The multitaper estimator $\hspdmt$ is calculated with $W = \caro^{-1/6}$ (see Theorem \ref{th_mse}) for a $1$-periodic $C^2$ spectral density $\spd$.
}
\end{figure}

To evaluate Theorem \ref{th_specwin}, we consider random fields on a grid $\gridn$ with $N = 128$ and $d = 2$.
The mask is a disk of radius $R$ given by $\sset = \{ q \in \gridn : \abs{q-(N/2, N/2)} < R\}$ and the bandwidth is fixed at $W = 1/4$.
For each radius, we then compute a set of proxy tapers $\randm_0, \ldots, \randm_{K-1}$ and calculate their spectral window $\rho$, which we compare to the ``ideal'' window $W^{-d} 1_{[-W/2,W/2]^d}$.
The result is shown in Figure \ref{fig:specwin_validation}.  
Computing the average slope in the logarithmic plot shows that the error decays approximately as $R^{-1.0}$.
This is close to the decay predicted by Theorem \ref{th_specwin}, whose leading term is $\pero W^{d-1} K^{-1} \approx \pero \caro^{-1} \, W^{-1}$ and is proportional to $R^{-1} \, W^{-1}$.

We use the same set of masks to evaluate Theorem \ref{th_mse}.
The tapers are generated as before, but with $W = \caro^{-1/6}$, as recommended in Theorem \ref{th_mse}.
We simulate $M = 200$ images $\popr_1, \ldots, \popr_M$ by sampling a random field with a $1$-periodic $C^2$ spectral density $\spd$ given by
\begin{equation}
\spd(\xi) = \left( 1_{|\xi| < \frac{1}{8}} * 1_{|\xi| < \frac{1}{8}} * 1_{|\xi| < \frac{1}{8}} \right)(\xi) = \int_0^\infty u^{-2} J_0(8|\xi| u) J_1(u)^3 du,
\end{equation}
where $J_\ell$ is the $\ell$th Bessel function.
Each image $\popr_\nu$ then yields a proxy multitaper spectral density estimate $\hspdr_\nu$, all of which are used to estimate the mean squared error as
\begin{equation}
\widehat{\mse}\left\{\hspdr(\xi)\right\} = \frac{1}{M} \sum_{\nu=1}^M \left| \hspdr_\nu(\xi) - \spd(\xi) \right|^2
\end{equation}
for all $\xi$ on an $N \times N$ grid.
These are then summarized by taking the maximum over $\xi$.
The result is shown, as a function of $R$, in Figure \ref{fig:mse_validation}.
Dividing the error by $\log^2 (\caro)$ and computing the average slope in the logarithmic plot, we obtain a decay of about $R^{-1.5}$.
Up to the logarithm factor and the norm of $\spd$, Theorem \ref{th_mse} gives an error bound of $\caro^{-2/3}$, which is proportional to $R^{-4/3}$ and therefore close to the empirical decay.

\subsection{Comparison with tensor Slepian tapers}
\label{sec_comp}
As discussed in Section \ref{sec_mt}, tensor products of Slepian functions may be used when $\sset$ is a subgrid of $\mathbb{Z}^d$.
To illustrate the performance of the proxy tapers, we therefore first compare them to standard tensor Slepian tapers defined on the subgrid shown in Figure \ref{fig:subgrid}.

\begin{figure}[t]
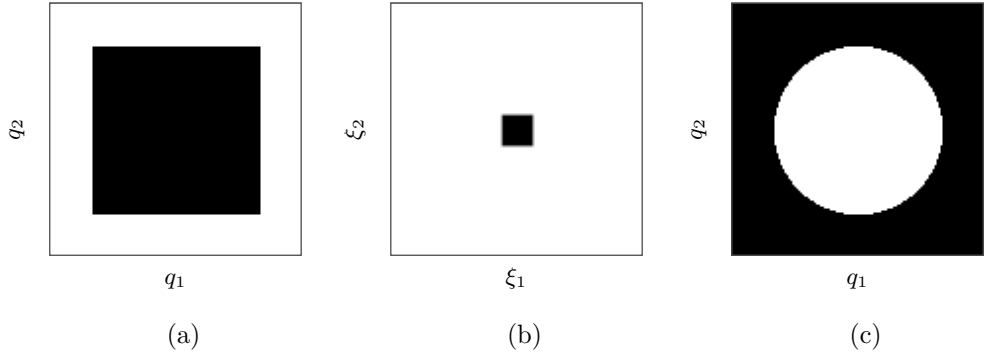

\centering
\begin{subfigure}{0.25\textwidth}
\input{gplt_scripts/recmask2_gplt}
\caption{}
\label{fig:subgrid}
\end{subfigure}
\begin{subfigure}{0.25\textwidth}
\input{gplt_scripts/mask1_gplt}
\caption{}
\label{fig:freqdomain}
\end{subfigure}
\begin{subfigure}{0.25\textwidth}
\input{gplt_scripts/mask2_gplt}
\caption{}
\label{fig:diskcompl}
\end{subfigure}
\caption{%
(a) A subgrid domain (black) of size $85$-by-$85$ within a larger domain (white) of size $128$-by-$128$.
(b) The target frequency profile (black) corresponding to $W = 1/8$.
(c) An irregular domain: the complement of a disk (black) with radius $43$ inside a $128$-by-$128$ square (white).
}
\end{figure}

The top row Figure \ref{fig:tenrectap} shows a few tensor Slepian tapers defined on the subgrid domain of Figure \ref{fig:subgrid} with target frequency profile given by \ref{fig:freqdomain}, which corresponds to $W = 1/8$.
Below are proxy tapers defined on the same subgrid with the same frequency profile and calculated using Algorithm \ref{algo:rt} for $T = 2$.
Both sets of tapers are properly supported on given domain, but their appearance is quite different.
Nonetheless, their accumulated spectral windows, shown in Figures \ref{fig:teninten} and \ref{fig:recinten}, both agree well with the target profile in Figure \ref{fig:freqdomain}.

To evaluate the performance of these tapers in a multitaper estimation setting, we generate a Gaussian process with spectral density given by Figure \ref{fig:density} on a $128$-by-$128$-pixel square.
The estimated density of the tensor Slepian multitaper estimator is given in Figure \ref{fig:tenmt} while that of the proxy tapers for $T = 2$ iterations is in Figure \ref{fig:recmt}.
Both agree quite well with the true density.
Indeed, the normalized root mean squared errors $\|\hspdmt-\spd\|/\|\spd\|$ and $\|\hspdr-\spd\|/\|\spd\|$ for the tensor Slepian tapers and proxy tapers are both approximately $1.67 \cdot 10^{-1}$.
The deviation between the two estimators is $\|\hspdmt-\hspdr\|/\|\hspdmt\| \approx 3.05 \cdot 10^{-3}$.
If we increase the number of iterations to $T = 72$, we obtain equality up to machine precision with $\|\hspdmt-\hspdr\|/\|\hspdmt\| \approx 5.42 \cdot 10^{-16}$.

\begin{figure}[t]
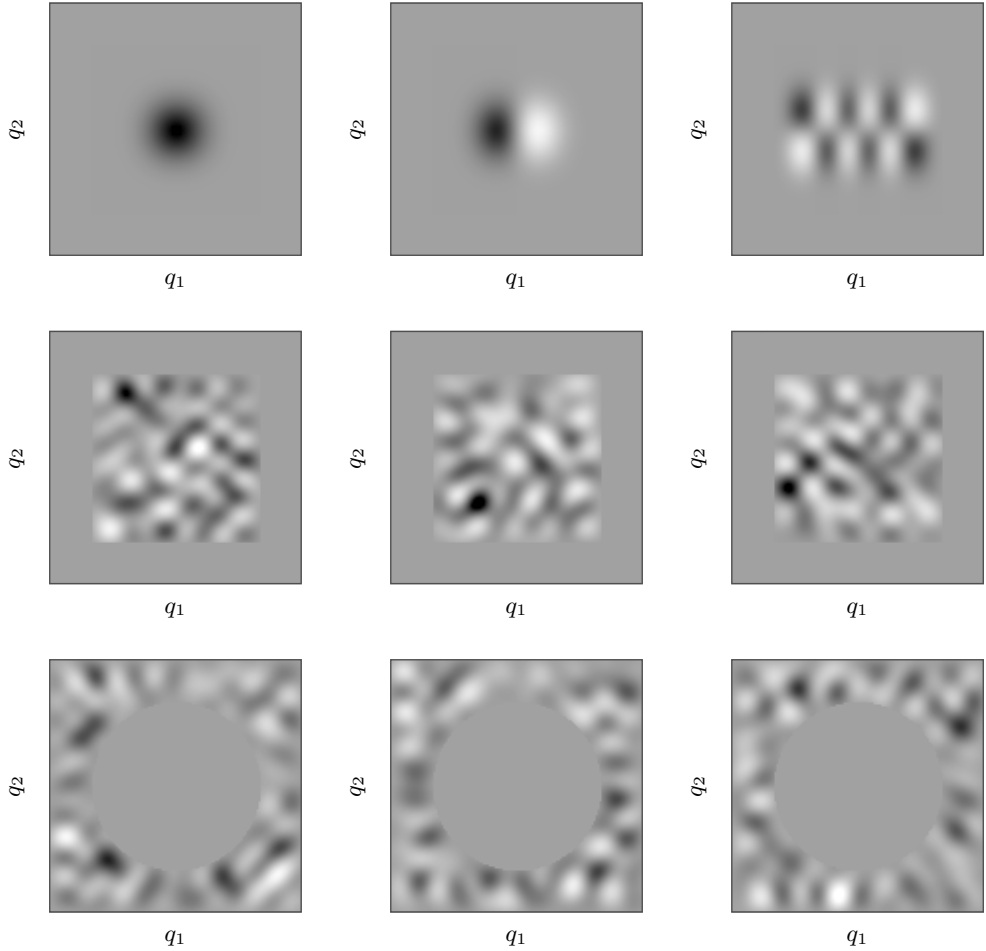

\centering
\begin{subfigure}{0.25\textwidth}
\input{gplt_scripts/tentap1_gplt}
\end{subfigure}
\begin{subfigure}{0.25\textwidth}
\input{gplt_scripts/tentap2_gplt}
\end{subfigure}
\begin{subfigure}{0.25\textwidth}
\input{gplt_scripts/tentap17_gplt}
\end{subfigure}

\begin{subfigure}{0.25\textwidth}
\input{gplt_scripts/rectap1_gplt}
\end{subfigure}
\begin{subfigure}{0.25\textwidth}
\input{gplt_scripts/rectap2_gplt}
\end{subfigure}
\begin{subfigure}{0.25\textwidth}
\input{gplt_scripts/rectap17_gplt}
\end{subfigure}

\begin{subfigure}{0.25\textwidth}
\input{gplt_scripts/tap1_gplt}
\end{subfigure}
\begin{subfigure}{0.25\textwidth}
\input{gplt_scripts/tap2_gplt}
\end{subfigure}
\begin{subfigure}{0.25\textwidth}
\input{gplt_scripts/tap17_gplt}
\end{subfigure}
\caption{(top row) Three tensor Slepian functions defined on the subgrid domain of Figure \ref{fig:subgrid}.
Three proxy tapers defined on (middle row) the domain of Figure \ref{fig:subgrid} and (bottom row) the disk complement domain of Figure \ref{fig:diskcompl}.}
\label{fig:tenrectap}
\end{figure}

\begin{figure}[t]
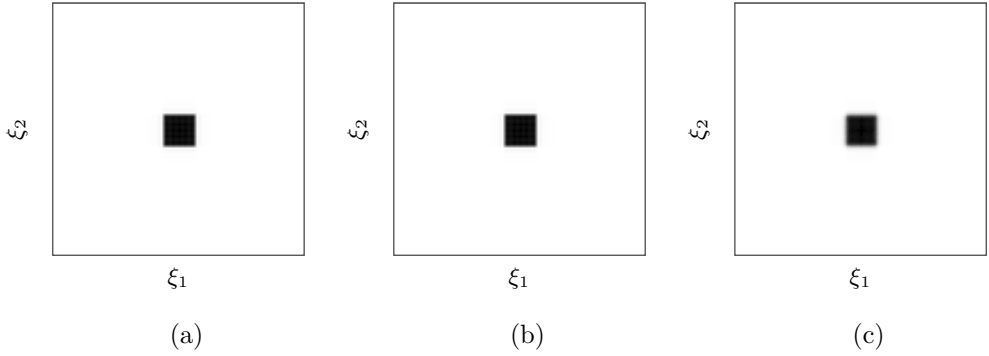

\centering
\begin{subfigure}{0.25\textwidth}
\input{gplt_scripts/teninten_gplt}
\caption{}
\label{fig:teninten}
\end{subfigure}
\begin{subfigure}{0.25\textwidth}
\input{gplt_scripts/recinten_gplt}
\caption{}
\label{fig:recinten}
\end{subfigure}
\begin{subfigure}{0.25\textwidth}
\input{gplt_scripts/inten_gplt}
\caption{}
\label{fig:inten}
\end{subfigure}
\caption{The accumulated spectral windows (black) of (a) the tensor Slepian tapers, (b) the proxy tapers for the subgrid domain of Figure \ref{fig:subgrid}, and (c) the proxy tapers for the irregular domain in Figure \ref{fig:diskcompl}.}
\end{figure}

\begin{figure}[t]
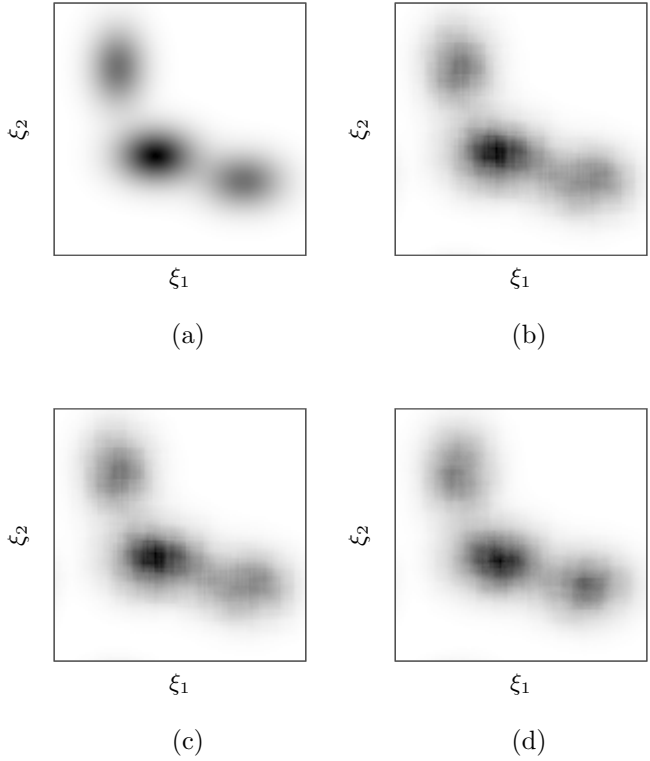

\centering
\begin{subfigure}{0.25\textwidth}
\input{gplt_scripts/density_gplt}
\caption{}
\label{fig:density}
\end{subfigure}
\begin{subfigure}{0.25\textwidth}
\input{gplt_scripts/tenmt_gplt}
\caption{}
\label{fig:tenmt}
\end{subfigure}

\begin{subfigure}{0.25\textwidth}
\input{gplt_scripts/recmt_gplt}
\caption{}
\label{fig:recmt}
\end{subfigure}
\begin{subfigure}{0.25\textwidth}
\input{gplt_scripts/mt_gplt}
\caption{}
\label{fig:mt}
\end{subfigure}
\caption{(a) The spectral density of a two-dimensional stochastic process. The density from a $128$-by-$128$ realization of the process using (b) tensor Slepian tapers defined on the subgrid of Figure \ref{fig:subgrid}, (c) proxy Slepian tapers defined on the same subgrid, and (d) proxy tapers defined on the disk complement of Figure \ref{fig:diskcompl}.}
\end{figure}

\subsection{Illustration for irregular domain}
We now replace the subgrid domain with the disk complement domain shown in Figure \ref{fig:diskcompl}.
Following the discussion of Section \ref{sec_mt}, it is computationally challenging to solve the eigenvalue \eqref{eq_eig} for the desired tapers (although standard eigenvalue solvers may still provide useful tapers, as observed in Section \ref{sec_rt}).
We can, however, compute proxy tapers over this domain using Algorithm \ref{algo:rt} with $T = 2$.
A few sample tapers are shown in the bottom row Figure \ref{fig:tenrectap}.
Again, their appearance is quite different from the Slepian tapers in the top row, but their accumulated spectral window shown in Figure \ref{fig:inten} agrees well with the target of Figure \ref{fig:freqdomain}.

Applying these tapers to estimate the spectral density of Figure \ref{fig:density} from one realization gives the density depicted in Figure \ref{fig:mt}.
The normalized root mean squared error is approximately $1.34 \cdot 10^{-1}$.

\subsection{Cryo-EM: Synthetic data}

\begin{figure}[ht]
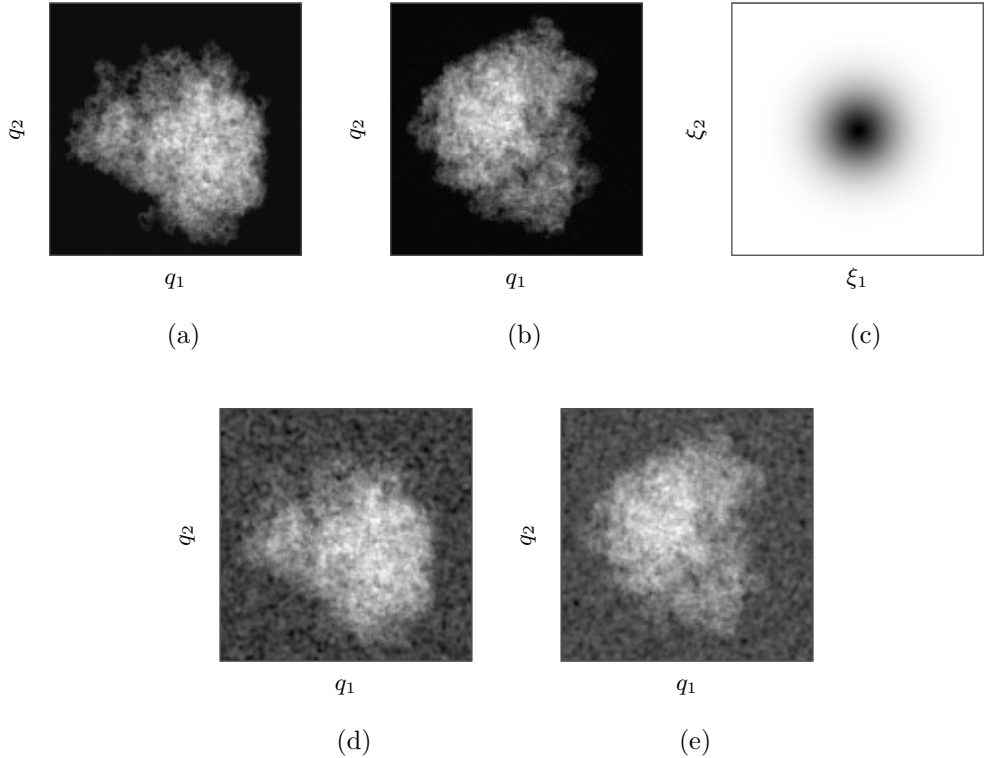

\centering
\begin{subfigure}{.25\textwidth}
\input{gplt_scripts/cryo_sim_sig1_gplt}%
\caption{}
\label{fig:cryo_sim_sig1}
\end{subfigure}
\begin{subfigure}{.25\textwidth}
\input{gplt_scripts/cryo_sim_sig2_gplt}%
\caption{}
\label{fig:cryo_sim_sig2}
\end{subfigure}
\begin{subfigure}{.25\textwidth}
\input{gplt_scripts/cryo_sim_psd_gplt}
\caption{}
\label{fig:cryo_sim_psd}
\end{subfigure}

\begin{subfigure}{.25\textwidth}
\input{gplt_scripts/cryo_sim_sig_noise1_gplt}%
\caption{}
\label{fig:cryo_sim_sig_noise1}
\end{subfigure}
\begin{subfigure}{.25\textwidth}
\input{gplt_scripts/cryo_sim_sig_noise2_gplt}%
\caption{}
\label{fig:cryo_sim_sig_noise2}
\end{subfigure}

\caption{%
Simulation of cryo-EM images.
(a,b) The clean projection images obtained from the density map of a 70S ribosome.
(c) The spectral density $\popr$ of the noise. (d,e) The projection images combined with noise generated using the spectral density.
}
\end{figure}

To evaluate our approach in a real-world application, we consider the estimation of noise power spectra in cryo-EM.
In cryo-EM imaging, a solution containing macromolecules of interest are frozen in a thin layer of vitreous ice which is then exposed to an electron beam.
A sensor records the transmitted electrons, resulting in a set of tomographic projections depicting the molecules from various viewing angles \cite{fr06}.
To reduce specimen damage, the electron dose is kept low, resulting in exceptionally noisy images, with noise power often exceeding that of the signal by a factor of ten or more.
Three-dimensional reconstruction of the molecules, the goal of cryo-EM, therefore requires a good characterization of the noise model.
This is especially important for methods which estimate the covariance structure of the underlying images \cite{bhzhsi16,ansi18}.
Since the noise characteristics vary with microscope configuration, ice thickness, and other experimental factors, we cannot rely on ensemble averages for low-variance estimation of the noise power spectrum.

To evaluate our approach for the cryo-EM application, we generate a number of projection images using a 70S ribosome density map on a grid with $N = 128$.
These are shown in Figures \ref{fig:cryo_sim_sig1} and \ref{fig:cryo_sim_sig2}.
We then generate Gaussian noise with a spectral density $\popr$ given in Figure \ref{fig:cryo_sim_psd}.
Although the noise in cryo-EM projections is dominated by Poisson-distributed shot noise \cite{bagr2009,vura2013}, the number of registered electrons is large enough that a Gaussian approximation suffices for our purposes.
The power spectral density was chosen to approximate those found in experimental datasets \cite{ansi17}, with high energy in the low frequencies that decays quickly to zero at high frequencies.
Adding the noise to the simulated projections, we obtain the images shown in Figure \ref{fig:cryo_sim_sig_noise1} and \ref{fig:cryo_sim_sig_noise2}.

\begin{figure}[ht]
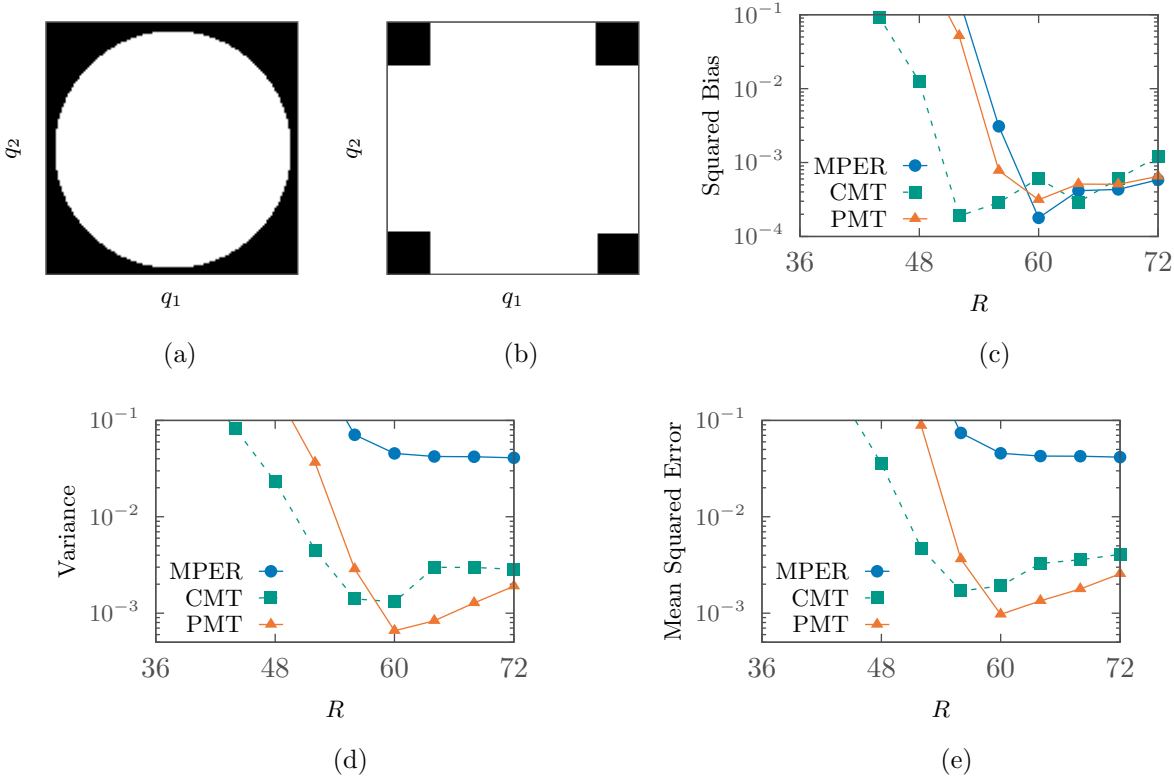

\centering
\begin{subfigure}{.25\textwidth}
\input{gplt_scripts/cryo_sim_mask_gplt}
\caption{}
\label{fig:cryo_sim_mask}
\end{subfigure}
\begin{subfigure}{.25\textwidth}
\input{gplt_scripts/cryo_sim_corner_mask_gplt}
\caption{}
\label{fig:cryo_sim_corner_mask}
\end{subfigure}
\begin{subfigure}{.45\textwidth}
\input{gplt_scripts/cryo_sim_bias_gplt}
\caption{}
\label{fig:cryo_sim_bias}
\end{subfigure}

\begin{subfigure}{.45\textwidth}
\input{gplt_scripts/cryo_sim_variance_gplt}
\caption{}
\label{fig:cryo_sim_variance}
\end{subfigure}
\begin{subfigure}{.45\textwidth}
\input{gplt_scripts/cryo_sim_mse_gplt}
\caption{}
\label{fig:cryo_sim_mse}
\end{subfigure}

\caption{
(a) The mask $\sset$ (black) for $R = 60$.
(b) The inscribed rectangular $\ssetgrids$ (black) for the same $R$.
(c) The squared bias of the masked periodogram (MPER), tensor multitaper on corners (CMT), and proxy multitaper (PMT) estimators.
(d) The variance of the estimators.
(e) The mean squared error of the estimators.
}
\end{figure}

Since the central disk of the images contains both the projected molecular density and the noise, we would like to estimate $\spd$ outside of this disk.
Specifically, we would like to restrict our estimator to a set $\sset$ of samples like the one shown in Figure \ref{fig:cryo_sim_mask}.
A common approach is to define the mask taper
\begin{equation*}
m[q] = \left\{ \begin{array}{ll} \caro^{-1/2}, & \mbox{if}~q \in \sset, \\ 0, & \mbox{otherwise} \end{array} \right.
\end{equation*}
and use \eqref{eq_mt} to calculate the tapered periodogram $\hspd_m(\xi)$ \cite{zhsi13,bhzhsi16}.
As only $K = 1$ taper is used, the result has high variance, so the estimates are averaged over a set of images to obtain an adequate estimate.
However, this fails to account for any variability in the noise models of the individual images.

A better estimator is obtained by replacing $\sset$ with a union of rectangular subgrids $\ssetgrids \subset \sset$
\begin{equation}
\label{eq_ssetgrids_def}
\ssetgrids := \bigcup_{a \in \{0, 1\}^2} \left\{ q \in \gridn : \abs{q_1 - Na_1} < \frac{N}{2}-\frac{R}{\sqrt{2}} \mbox{~and~} \abs{q_2 - Na_2} < \frac{N}{2} - \frac{R}{\sqrt{2}} \right\}.
\end{equation}
We then apply a standard tensor multitaper estimator to each subgrid and average the results.
These tapers are tensor products of one-dimensional Slepian sequences, as described in Section \ref{sec_mt}.
The $\ssetgrids$ corresponding to the $\sset$ of Figure \ref{fig:cryo_sim_mask} is shown in Figure \ref{fig:cryo_sim_corner_mask}.
Depending on the geometry of $\sset$, $\ssetgrids$ may discard many points in $\sset$, increasing variance of the estimate.
For small $K$, this also increases bias, as fewer points in this regime leads to a wider accumulated spectral window $\win$.

\paragraph{Performance}
We now compare the performance of these baseline estimators, the tapered periodogram $\hspd_m$ and the tensor multitaper estimator $\hspdmt$, to that of our proposed estimator $\hspdr$.
For a given $R$, we define $\sset$ by \eqref{eq_sset_def} and $\ssetgrids$ by \eqref{eq_ssetgrids_def}.
When $R$ is low, we expect a large bias as the samples are contaminated by the projected density maps, while increasing $R$ results in lower bias but higher variance due to the lower number of available samples.
We estimate the bias, variance, and MSE of the three estimators by computing them on $M = 1000$ synthetic images of size $N = 128$ generated as described in the beginning of the section.
For the tensor multitaper and proxy taper estimators, we set $W = 1/8$.
We plot the resulting bias, variance, and MSE estimates as a function of $R$ in Figures \ref{fig:cryo_sim_bias}, \ref{fig:cryo_sim_variance}, and \ref{fig:cryo_sim_mse}.

The bias of the proxy multitaper estimator is higher than the tensor multitaper for low $R$.
This is due to the original mask $\sset$ overlapping with the support of the projection images to a greater extent than $\ssetgrids$ at these radii.
As $R$ increases above $60$, however, the bias for both estimators drops down to the same level.
Since it has $K = 1$, no smoothness is imposed on the tapered periodogram which therefore has lower bias compared to the multitaper estimators.

At low $R$, the influence of the clean projection images yields high variance for all the estimators, since the images vary according to viewing angle.
The effect is exacerbated for the proxy multitaper estimators and the tapered periodogram since $\sset$ has greater overlap with the projections compared to $\ssetgrids$.
As $R$ increases, however, \emph{the proxy multitaper estimator $\hspdr$ enjoys a lower variance}, since it draws upon a larger number of samples compared to the tensor multitaper estimator $\hspdmt$, reducing the variance by a factor of two.
The tapered periodogram, meanwhile, has high variance for all $R$ since it only employs a single taper, providing no variance reduction.

The low bias and variance for high $R$ combine to yield a lower MSE for the proxy multitaper estimator $\hspdr$ compared to the other estimators.
On average, the proxy multitaper estimator gives an error a factor of $1.7$ lower than the tensor multitaper.

\begin{figure}[t]
\centering
\input{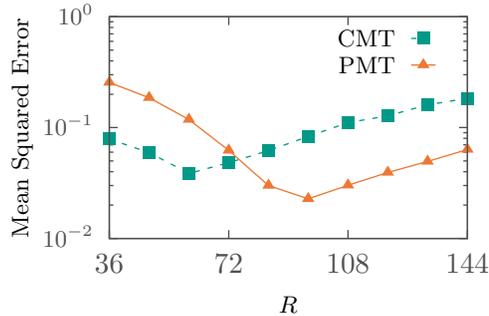}
\caption{
\label{fig:cryo_exp_mse}
    The mean squared errors (MSE's) of the masked periodogram (MPER), tensor multitaper on corners (CMT), and proxy multitaper (PMT) estimators on a subset of the EMPIAR-10028 cryo-EM dataset \cite{rib80s}.}
\end{figure}

\medskip

\subsection{Cryo-EM: Experimental data}
We now evaluate performance on images consisting of experimental projections of an 80S ribosome complex from the EMPIAR-10028 dataset \cite{rib80s}.
The images are defined on a grid with $N = 360$.
For our evaluation, we choose a subset of $M = 120$ images from the dataset and set $W = 1/16$.
Two sample images from this subset are shown in Figure \ref{fig:cryo_ex}.

To obtain a reasonable approximation of the true power spectrum for each projection image, we process the entire dataset through the RELION software package \cite{relion}.
This yields estimates of the underlying molecular density and allows us to simulate clean projection images for each of the noisy images.
By subtracting the estimated clean images from the noisy images, we obtain a set of images consisting mostly of noise.
We may then estimate their power spectra by applying a tensor multitaper estimator over the entire grid $\gridn$.
Since these estimates incorporate the noise in the center of the image, they should provide a good estimate of the spectral density of the noise in the whole image.
We shall therefore use these as a standard against which we compare the estimates obtained from the complement of the disk on the original projection images.

As before, we calculate the tapered periodogram $\hspd_m$ on $\sset$, the tensor multitaper estimator $\hspdmt$ on $\ssetgrids$, and the proxy multitaper estimator $\hspdr$ on $\sset$.
Each estimated power spectrum is compared to the ``ground truth'' power spectrum obtained for the corresponding image as described above, which gives an estimated MSE for each method.
Similarly to the simulation results, the error for the tapered periodogram $\hspd_m$ is dominated by the variance, resulting in an MSE of around $10$ largely independent of $R$.
The MSE's for the tensor multitaper and proxy multitaper estimators are plotted for each estimator in Figure \ref{fig:cryo_exp_mse}.

The tensor multitaper estimator performs well for low $R$, but increasing $R$ reduces the area of $\ssetgrids$, yielding higher variance and higher MSE.
A similar behavior is observed for the proxy multitaper estimates, but the reduction in area for $\sset$ is not as drastic for increasing $R$, so the variance remains small compared to the tensor multitaper estimator.
At $R \approx 96$, the error is minimized and the proxy multitaper outperforms the tensor multitaper by a factor of $1.7$.

\section{Conclusion}

We have analyzed the multitaper estimator on arbitrary acquisition domains, providing performance bounds on the mean squared error.
Furthermore, we show that the multitaper estimate only depends on the tapers through their linear span.
This explains the success of applying standard eigenvalue algorithms to the ill-posed Slepian eigenproblem, since a common mode of failure is for these to yield a set of vectors with the same span as the desired eigenvectors.
Using the resulting vectors as tapers therefore yields results close to those obtained with the true Slepian tapers.
We also present a more straightforward approach of calculating these proxy Slepian tapers using the block power method.
The performance of the resulting proxy multitaper estimator is shown to be comparable to that using Slepian tapers when these are available for rectangular domains.
We also illustrate the performance of the proxy multitaper estimator for more general domains and compared it to tensor Slepian multitaper estimators on rectangular subgrids.
Numerical results are obtained for both synthetic examples and on experimental data obtained from cryo-EM imaging.

Future directions of research involve adapting the factor analysis framework proposed for periodogram estimators \cite{ansi17} to multitaper estimators.
This would allow for greater variance reduction when a linear structure exists in the variability of spectral density between independent realizations of the random field.
Such a situation arises, for example, in the cryo-EM noise estimation task, where a set of underlying noise sources combine at arbitrary strengths to yield the noise process in a given image.

\section*{Acknowledgments}
The authors are very grateful to Lu\'is Daniel Abreu, Tomasz Hrycak, Frederik Simons, and Amit Singer, who motivated this article and provided valuable input.
They would also like to thank the anonymous reviewers for their helpful remarks.
J.~L.~R. gratefully acknowledges support from the Austrian Science Fund (FWF): P 29462 and Y 1199, and from the WWTF grant INSIGHT (MA16-053).
The Flatiron Institute is a division of the Simons Foundation.

\appendix

\section{Proofs}
\label{sec_estimates}

The $p$-norm of a vector $x \in \mathbb{R}^d$ is denoted $\abs{x}_p$. For a function $f \in \ell^1(\mathbb{Z}^d)$, we denote its $\ell^1$-norm $\sum_q |f[q]|$ by $|f|$.

For two non-negative functions $f, g: X \to [0, +\infty)$, we write $f \lesssim g$ if there exists a constant $C \geq 0$ such that $f \leq C g$. We also write $f \asymp g$, if $f \lesssim g$ and $g \lesssim f$.

Recall that a function $f: \mathbb{R}^d \to \mathbb{C}$ is \emph{1-periodic} if $f(x+k)=f(x)$ for all $k \in \mathbb{Z}^d$. The convolution of two 1-periodic functions $f, g$ is defined as
\begin{align}
\label{eq_conv}
f*g(x) = \int_{[-1/2,1/2]^d} f(y) g(x-y) dy.
\end{align}
The $L^\infty$-norm and $L^1$-norm of a measurable, 1-periodic function $f: \mathbb{R}^d \to \mathbb{C}$ are, respectively,
\begin{align*}
\norm{f}_{\infty} &:= \esssup_{x \in [-1/2,1/2]^d} \abs{f(x)},\\
\norm{f}_{1} &:= \int_{[-1/2,1/2]^d} \abs{f(x)} dx.
\end{align*}
For clarity, we sometimes write $\norm{f}_{L^\infty([-1/2,1/2]^d)}$ and $\norm{f}_{L^1([-1/2,1/2]^d)}$.

\subsection{Convolution estimates}
\label{sec_conv}

First, we note that
\begin{align}
\label{eq_loi}
\norm{f*g}_\infty \leq \norm{f}_\infty \norm{g}_{L^1([-1/2,1/2]^d)}.
\end{align}
Furthermore, for a parameter $W \in (0,1/2)$ we consider the periodic extension of the normalized characteristic function $W^{-d} 1_{[-W/2,W/2]^d}$, and, by a slight abuse of notation, we define its convolution with a 1-periodic function $f$ as
\begin{align*}
f * W^{-d} 1_{[-W/2,W/2]^d} (x) = \frac{1}{W^{d}} \int_{[-W/2,W/2]^d} f(x-y) dy.
\end{align*}
An estimate on a second-order Taylor expansion shows that
\begin{align}
\label{eq_reg}
\norm{f- \tfrac{1}{W^{d}} f * 1_{[-W/2,W/2]^d}}_{\infty}
\lesssim \norm{f}_{C^2} W^2.
\end{align}

\subsection{Trace and norm of the Toeplitz operators}
The estimates derived in this section are also relevant in numerical analysis of Fourier extensions.
In particular, they improve on the results in \cite{mahu17} by avoiding assumptions on the (digital) topology
of the set $\sset$. See also \cite{ni01} for related estimates.
\begin{proposition}
\label{prop_reg}
Let $a, b \in \ell^1(\Zdst)$ with $\sum_q a[q] = 1$. Then
\begin{align*}
\norm{a*b-b}_{\ell^1(\Zdst)} \leq \norm{\nabla b}_1
\sum_q \abs{q}_\infty \abs{a(q)},
\end{align*}
where
$\left(\nabla_k b\right)[q] := b[q+e_k]-b[q]$ for $k = 1, \ldots, d$,
$\norm{\nabla b}_1 = \sum_{k=1}^d \norm{\nabla_k b}_1$, and
$\abs{q}_\infty = \max\{\abs{q_1}, \ldots, \abs{q_d} \}$.
\end{proposition}
\begin{proof}
For $q \in \Zdst$ and $k\in\{1,\ldots,d\}$ we let $\chi_k q \in \Zdst$ be the truncated vector
$(\chi_k q)_j = q_j 1_{j \leq k}$, and also $\chi_0 q=0$. For $q' \in \Zdst$, let us write
\begin{align*}
&(a*b)[{q'}]-b[{q'}]=\sum_{q \in \Zdst} \left(b[{q+q'}]-b[{q'}]\right)a[{-q}]
\\
&\qquad=\sum_{q \in \Zdst} 
\sum_{k=1}^d
\left(b[{\chi_k q+q'}]-b[{\chi_{k-1}q+q'}]\right)a[{-q}]
\\
&\qquad=\sum_{q \in \Zdst} 
\sum_{k=1}^d
\left(b[{\chi_{k-1} q+q'+q_k e_k}]-b[{\chi_{k-1}q+q'}]\right)a[{-q}]
\\
&\qquad=\sum_{q \in \Zdst} 
\sum_{k=1}^d
\sum_{l=0}^{|q_k|-1}
\left(b[{\chi_{k-1} q+q'+\sgn(q_k)(l+1)e_k}]-b[{\chi_{k-1}q+q'+\sgn(q_k)l e_k}]\right)a[{-q}]
\\
&\qquad=\sum_{q \in \Zdst} 
\sum_{k=1}^d
\sum_{l=0}^{|q_k|-1}
\nabla_{k}b[{\chi_{k-1}q+q'+\sgn(q_k) l e_k}] a[{-q}].
\end{align*}
Therefore,
\begin{align*}
\norm{a*b-b}_1&\leq
\sum_{q \in \Zdst} 
\sum_{k=1}^d
\sum_{l=0}^{|q_k|-1}
\sum_{q' \in \Zdst}
\abs{\nabla_{k}b[{\chi_{k-1}q+q'+\sgn(q_k) l e_k}]} \abs{a[{-q}]}
\\
&\leq
\sum_{q \in \Zdst} 
\sum_{k=1}^d
\norm{\nabla_{k}b}_1 \abs{q_k} \abs{a[{-q}]}
\leq
\norm{\nabla b}_1
\sum_{q \in \Zdst} \abs{q}_\infty \abs{a[q]},
\end{align*}
as desired, since $|q_k| \le \abs{q}_\infty$.
(See \cite{ni01,abgrro16} for related estimates.)
\end{proof}

\begin{proposition}
\label{prop_tr}
Let $\mat$ be the matrix in \eqref{eq_mat}. Then
\begin{align}
\label{eq_tn}
\trace\left[\mat\right]-\trace\left[({\mat})^2\right]
\lesssim \pero W^{d-1}\left[ 1+ \log\left( \frac{\caro}{\pero} \right) \right].
\end{align}
\end{proposition}
\begin{proof}

\noindent \emph{Step 1. (Computations)}. The function
\begin{align}
\label{eq_h}
\hh[q] := \prod_{j=1}^d \frac{\sin (\pi W q_j)}{\pi q_j} = W^d \sinc_d(W q), \qquad q = (q_1, \ldots, q_d) \in \Zdst,
\end{align}
determines the Fourier series
\begin{align*}
\sum_{q \in \Zdst} \hh[q] \e{\ip{q}{\xi}}= 1_{[-W/2,W/2]^d}(\xi), \qquad \xi \in [-W/2, W/2]^d.
\end{align*}

Note that, in terms of $\hh$, \eqref{eq_mat} reads:
$\mat_{q,q'} = 1_\sset(q) \hh[q-q'] 1_\sset(q')$.
We first compute
\begin{align}
\trace \left[\mat\right]&=\sum_{q \in \sset} \hh[{q-q}] = \caro \hh_0 = \caro
\norm{1_{[-W/2,W/2]^d}}_1
\\ &=
\label{eq_trmat}
W^d \caro = W^d \sum_{q\in\Zdst} 1_\sset[q].
\end{align}
Second,
\begin{align}
\trace \left[({\mat})^2\right] &=
\sum_{q,q' \in \mathbb{Z}^d} 1_\sset[q] \abs{\hh[{q-q'}]}^2 1_\sset[q']
\\
\label{eq_trmat2}
&= W^d \sum_{q \in \mathbb{Z}^d}
\left(1_\sset * a\right)[q] 1_\sset[q],
\end{align}
where \[a[q] := W^{-d} \abs{\hh[q]}^2.\]

\noindent \emph{Step 2. (Truncation errors)}. Let $L>0$ and consider the function
\[\tilde{a}[q] := a[q] 1_{\abs{q} \leq L}.\]

We claim that, for $L \geq 2$,
\begin{align}
\label{eq_aaa}
&\norm{a-\tilde{a}}_1 = W^{-d} \sum_{q \in \Zdst, \abs{q} > L} \abs{\hh[q]}^2 \lesssim (WL)^{-1},
\\
\label{eq_bbb}
&\sum_{q \in \Zdst} \abs{q} \abs{\tilde{a}[q]} = W^{-d} \sum_{q \in \Zdst, \abs{q} \leq L} 
\abs{q}\abs{\hh[q]}^2
\lesssim \frac{\log L}{W}.
\end{align}
To show these estimates, we write $\hh[q] = \rr[q_1] \ldots \rr[q_d]$ with
\begin{align*}
\rr[q_1]=\frac{\sin (\pi W q_1)}{\pi q_1} = W \sinc (W q_1), \qquad q_1 \in \mathbb{Z}.
\end{align*}
We first note the following:
\begin{align*}
&\sum_{q_1 \in \mathbb{Z}} \abs{\rr[q_1]}^2 = \bignorm{1_{[-W/2,W/2]}}^2_{L^2([-1/2,1/2])} = W,
\\
&\sum_{q_1 \in \mathbb{Z}, \abs{q}>L} \abs{\rr[q_1]}^2
\lesssim \sum_{q_1 \in \mathbb{Z}, \abs{q_1}>L} \frac{1}{\abs{q_1}^2} \lesssim \frac{1}{L},
\\
&\sum_{q_1 \in \mathbb{Z}, \abs{q_1} \leq L} \abs{q_1} \abs{\rr[q_1]}^2
\lesssim \sum_{q_1 \in \mathbb{Z}, \abs{q_1} \leq L} \frac{1}{\abs{q_1}} \lesssim \log(L),
\qquad L \geq 2.
\end{align*}
To show \eqref{eq_aaa}, we exploit the fact that the $d$-ball of radius $L$ contains a $d$-cube of side $2 \cdot d^{-1/2} L$ and estimate
\begin{align*}
\sum_{q \in \Zdst, \abs{q} > L} \abs{\hh[q]}^2
&\leq \sum_{k=1}^d \sum_{\stackrel{q \in \Zdst}{\abs{q_k} > d^{-1/2} L}} \abs{\hh[q]}^2
\\
&= d \, \Big( \sum_{q_1} \abs{\rr[q_1]}^2 \Big)^{d-1}
\Big(\sum_{\abs{q_1}>d^{-1/2} L} \abs{\rr[q_1]}^2 \Big)
\\
&
\lesssim W^{d-1} L^{-1}.
\end{align*}
Similarly, for $L \geq 2$ 
\begin{align*}
\sum_{q \in \Zdst, \abs{q} \leq L} \abs{q}\abs{\hh[q]}^2
&\leq
\sum_{k=1}^d \sum_{q \in \Zdst, \abs{q} \leq L} \abs{q_k}\abs{\hh[q]}^2
\leq
\sum_{k=1}^d \sum_{q \in \Zdst, \abs{q_k} \leq L} \abs{q_k}\abs{\hh[q]}^2
\\
&= d \, \Big(\sum_{q_1\in\Zst} \abs{\rr[q_1]}^2 \Big)^{d-1}
\Big(\sum_{\abs{q_1} \leq L} \abs{q_1}\abs{\rr[q_1]}^2 \Big)
\lesssim W^{d-1} \log(L),
\end{align*}
which gives \eqref{eq_bbb} and establishes the remaining claim.

\noindent \emph{Step 3. (Final estimates)}. 
Let $b := \norm{\tilde{a}}_1^{-1} \tilde{a}$. Noting that $\sum_q b[q] = 1$,
we may combine \eqref{eq_trmat} with \eqref{eq_trmat2} and use Proposition \ref{prop_reg} to form the bound
\begin{align}
\nonumber
&\trace[\mat]-\trace[({\mat})^2] \leq W^d \norm{(1_\sset*a)1_\sset - 1_\sset}_1 \\
\nonumber
&\qquad \leq W^d \big( \norm{1_\sset*(a-b)1_\sset}_1 + \norm{(1_\sset*b)1_\sset - 1_\sset}_1\big) \\
\nonumber
&\qquad \leq W^d \big( \norm{1_\sset*(a-b)}_1 + \norm{1_\sset*b - 1_\sset}_1\big) \\
\nonumber
&\qquad \leq W^d
\left( \norm{1_\sset}_1 \norm{a-b}_1 + \norm{\nabla 1_\sset}_1 \sum_{q\in\Zdst} |q| |b[q]| \right) \\
\label{eq_tr_bound_1}
&\qquad = W^d \left( \caro \norm{a-b}_1 + \pero \sum_{q\in\Zdst} \abs{q} \abs{b[q]} \right),
\end{align}
where we have also used the fact that $\norm{1_\sset*(a-b)}_1 \leq \norm{1_\sset}_1 \norm{a-b}_1$.

Let $\varepsilon := \norm{a-\tilde{a}}_1$.
Since $\norm{a}_1 = 1$ and $\tilde{a}[q] \le a[q]$ for all $a \in \Zst$, we have $\varepsilon = 1 - \norm{\tilde{a}}_1$.
Then, by \eqref{eq_aaa}, $\varepsilon \lesssim (WL)^{-1}$. Hence, there exist a constant $C_0>2$ such that
$\varepsilon < 1/2$, if $WL \geq C_0$. Let us assume that for the moment that $WL \geq C_0$, 
so that $\varepsilon < 1/2$ and $\norm{\tilde{a}}_1 = 1 - \varepsilon > 1/2$, and 
estimate
\begin{align*}
\norm{a-b}_1 &\leq \norm{a-\tilde{a}}_1 + \norm{\tilde{a}-b}_1=
\varepsilon + \big(\norm{\tilde{a}}_1^{-1}-1\big) \norm{\tilde{a}}_1 
\\
&= \varepsilon + (1-\norm{\tilde{a}}_1)
= 2\varepsilon \lesssim (WL)^{-1}.
\end{align*}
Similarly, by \eqref{eq_bbb},
\begin{align*}
\sum_{q\in\Zdst} \abs{q} \abs{b[q]} = \norm{\tilde{a}}^{-1}_1
\sum_{q\in\Zdst} \abs{q} \abs{\tilde{a}[q]} \leq 2 
\sum_{q\in\Zdst} \abs{q} \abs{\tilde{a}[q]}
\lesssim W^{-1} \log(L).
\end{align*}
Substituting these estimates into \eqref{eq_tr_bound_1} gives
\begin{align*}
&\trace[\mat]-\trace[({\mat})^2]
\lesssim W^{d-1} \left(\caro L^{-1} + \pero \log L\right).
\end{align*}
The right-hand side is minimized at $L=\tfrac{\caro}{\pero}$, which yields \eqref{eq_tn}, provided that
$WL = W \tfrac{\caro}{\pero} \geq C_0$. On the other hand, if 
$W \tfrac{\caro}{\pero} \leq C_0$, \eqref{eq_tn} is trivially 
true because 
\[\trace [{\mat}]= W^d \caro = W^{d-1} W \caro \leq C_0 W^{d-1}\pero.\]

\end{proof}

\subsection{Expectation and variance of the tapered estimators}
\begin{proof}[Proof of Proposition \ref{prop_mt_general}]
The fact that $\win$ depends only on the linear span the tapers follows from the proof of Proposition \ref{prop_span} below.
The variance bound can be proved as in \cite[Theorem 2]{liro08}. Only the orthogonality of the tapers is 
important here (see also \cite[Chapter 3]{hola12}).
For the bias of a taper $m \in \ell^1(\Zdst)$,
a direct calculation yields:
\begin{align*}
\mathbb{E} \left\{ \hspd_{m}(\xi) \right\}
=
\left( \abs{M}^2 * \spd \right) (\xi),
\mbox{ where }
M(\xi) = \sum_{q \in \Zdst} m[q] \e{\ip{q}{\xi}}.
\end{align*}
By averaging this expression over all tapers, we then obtain the bias of the multitaper estimator.
\end{proof}

\subsection{Description of the spectral window}
\begin{proof}[Proof of Theorem \ref{th_specwin}]
We use the notation $M_k(\xi) = \sum_{q \in \Zdst} m_k[q] \e{\ip{q}{\xi}}$.

\noindent {\bf Step 1}. We use that $\{m_0, \ldots, m_{\caro-1}\}$ is an orthonormal basis
of $\ell^2(\Omega)$ and compute
\begin{align*}
\win(\xi) &= \frac{1}{K}\sum_{k=0}^{K-1} \abs{M_k(\xi)}^2 \leq
\frac{1}{K} \sum_{k=0}^{\caro-1} \abs{M_k(\xi)}^2
\\
    &=\frac{1}{K} \sum_{k=0}^{\caro-1} \abs{\sum_{q\in\sset} m_k[q] \e{\ip{q}{\xi}}}^2
=
\frac{1}{K} \sum_{q \in \Omega} \abs{\e{\ip{q}{\xi}}}^2=\frac{\caro}{K},
\end{align*}
since an orthonormal change of basis preserves the norm of the function.
We then note that
\begin{align*}
\lambdak =
\lambdak \norm{m_k}_2^2=
\sum_{q,q' \in \Zdst} \overline{m_k[q]} \mat[q,q'] m_k[q']
 = \int_{[-W/2,W/2]^d}\abs{M_k(\xi)}^2 d\xi.
\end{align*}
This lets us form the estimate
\begin{align*}
&\int_{[-W/2,W/2]^d} \abs{\win(\xi) - \frac{\caro}{K} 1_{[-W/2,W/2]^d}(\xi)} d\xi
\\
&\qquad
=
\frac{\caro}{K} \int_{[-W/2,W/2]^d} 1_{[-W/2,W/2]^d}(\xi) d\xi - \int_{[-W/2,W/2]^d} \win(\xi) d\xi
\\
&\qquad
=
\frac{\caro W^d}{K} - \frac{1}{K}\sum_{k=0}^{K-1} \int_{[-W/2,W/2]^d}\abs{M_k(\xi)}^2 d\xi
= \frac{\caro W^d}{K} - \frac{1}{K}\sum_{k=0}^{K-1} \lambdak
\leq 1-\frac{1}{K}\sum_{k=0}^{K-1} \lambdak.
\end{align*}
Similarly,
\begin{align*}
&\int_{[-1/2,1/2]^d\setminus[-W/2,W/2]^d} \abs{\win(\xi) - \frac{\caro}{K} 1_{[-W/2,W/2]^d}(\xi)} d\xi
\\
&\qquad
=
\int_{[-1/2,1/2]^d\setminus[-W/2,W/2]^d} \win(\xi) d\xi
\\
&\qquad
=
\frac{1}{K}\sum_{k=0}^{K-1}
\int_{[-1/2,1/2]^d\setminus[-W/2,W/2]^d}\abs{M_k(\xi)}^2 d\xi
\\
&\qquad
=
\frac{1}{K}\sum_{k=0}^{K-1}
\left( 1 - \int_{[-W/2,W/2]^d} \abs{M_k(\xi)}^2 d\xi \right)
= \frac{1}{K}\sum_{k=0}^{K-1} (1-\lambdak)
= 1-\frac{1}{K}\sum_{k=0}^{K-1} \lambdak.
\end{align*}
Hence,
\begin{align}
\label{eq_xx}
\bignorm{\win - \tfrac{\caro}{K} 1_{[-W/2,W/2]^d}}_1
\lesssim 1-\frac{1}{K}\sum_{k=0}^{K-1} \lambdak.
\end{align}

\noindent {\bf Step 2}.
Using \eqref{eq_xx}, we estimate,
\begin{align}
\nonumber
&\Bignorm{\win - \frac{1}{W^d} 1_{[-W/2,W/2]^d}}_{L^1([-1/2,1/2]^d)}
\\
\nonumber
&\qquad \leq
\Bignorm{\win - \frac{\caro}{K} 1_{[-W/2,W/2]^d}}_{L^1([-1/2,1/2]^d)}
+ \Bignorm{\left(\frac{1}{W^d} - \frac{\caro}{K} \right) 1_{[-W/2,W/2]^d}}_{L^1([-1/2,1/2]^d)}
\\
\label{eq_step2}
&\qquad =
\Bignorm{\win - \frac{\caro}{K} 1_{[-W/2,W/2]^d}}_{L^1([-1/2,1/2]^d)}
+ \frac{\abs{K-W^d\caro}}{K}
\lesssim 1-\frac{1}{K}\sum_{k=0}^{K-1} \lambdak
+\frac{1}{K}.
\end{align}

{\bf Step 3}. We now proceed as in \cite{abro17}:
\begin{align*}
&\trace[\mat]-\trace[({\mat})^2] = \trace[\mat(\mathrm{I} - \mat)] \\
&\qquad
= \sum_{k=0}^{\caro-1}\lambdak(1-\lambdak)
\\
&\qquad =\sum_{k=0}^{K-1}\lambdak(1-\lambdak)+\sum_{k=K}^{\caro-1}\lambda_{k}(1-\lambdak) \\
&\qquad \geq \lambda
_K\sum_{k=0}^{K-1}(1-\lambdak)+(1-\lambda_K)\sum_{k=K}^{\caro-1}\lambdak \\
&\qquad =\lambda_{K}
K-\lambda_{K}\sum_{k=0}^{K-1}\lambdak+\left(1-\lambda_{K}\right)
\left(W^d \caro-\sum_{k=0}^{K-1}\lambdak \right) \\
&\qquad =\lambda _{K}K+
W^d \caro
(1-\lambda _{K})-\sum_{k=0}^{K-1}\lambdak \\
&\qquad =
W^d \caro-\sum_{k=0}^{K-1}\lambdak+\lambda_{K}(K-W^d \caro) \\
&\qquad \geq K-\sum_{k=0}^{K-1}\lambdak-1.
\end{align*}
By Proposition \ref{prop_tr},
\begin{align*}
K-\sum_{k=0}^{K-1}\lambdak \lesssim
\pero W^{d-1}\left[ 1+ \log\left( \frac{\caro}{\pero} \right) \right].
\end{align*}
where constant vanishes since $\pero W^{d-1} \geq 1$.
Finally, we combine this with \eqref{eq_step2} to obtain \eqref{eq_con}.
\end{proof}

\subsection{Proof of Theorem \ref{th_mse}}
\begin{proof}
Using that $W \asymp \left(\tfrac{K}{\caro}\right)^{1/d}$, we 
invoke Theorem \ref{th_specwin}, \eqref{eq_reg} and \eqref{eq_loi}
to obtain
\begin{align*}
&\abs{\mathrm{Bias} \left\{ \hspdmt(\xi) \right\}} =
\abs{\spd(\xi)-\mathbb{E} \left\{ \hspdmt(\xi) \right\}}
\\
&\qquad \lesssim \abs{\spd(\xi) -
\left(\spd*\tfrac{1}{W^d}1_{[-W/2,W/2]^d}\right)(\xi)}
+ \abs{\spd*(\win-\tfrac{1}{W^d}1_{[-W/2,W/2]^d})(\xi)}
\\
&\qquad \lesssim \norm{\spd}_{C^2} W^2 + \norm{\spd}_\infty
\bignorm{\win-\tfrac{1}{W^d}1_{[-W/2,W/2]^d}}_{L^1([-W/2,W/2]^d)}
\\
&\qquad \lesssim \norm{\spd}_{C^2}
\left(
W^{2} +
\frac{\pero W^{d-1}}{K} \left[ 1+ \log\left( \frac{\caro}{\pero} \right) \right]
\right)
\\
&\qquad \lesssim \norm{\spd}_{C^2}
\left(
\frac{K^{2/d}}{\caro^{2/d}} +
\frac{\pero }{\caro^{1-1/d}K^{1/d}} \left[ 1+ \log\left( \frac{\caro}{\pero} \right) \right]
\right).
\end{align*}
Combining this with Proposition \ref{prop_mt_general}, we obtain
\begin{align*}
\mse \left\{\hspdmt (\xi)\right\} &= \left[\mathrm{Bias}
\left\{
\hspdmt(\xi) \right\} \right]^2 + \var \left\{\hspdmt \right\}(\xi)
\\
&\lesssim \norm{\spd}^2_{C^2}
\left(
\frac{K^{4/d}}{\caro^{4/d}} +
\frac{\pero^2 }{\caro^{2-2/d}K^{2/d}} \left[ 1+ \log\left( \frac{\caro}{\pero} \right) \right]^2
+ \frac{1}{K}
\right).
\end{align*}

Finally for $d=2$, if $K \approx \caro^{2/3}$, and
$\caro \leq \pero^2 \leq C \caro$, for some constant $C>0$, we obtain
\begin{align*}
\mse \left\{\hspdmt \right\}(\xi) &\lesssim
\left( \caro^{-2/3}+
\caro^{-2/3} \left[ 1+ \log\left( \caro^{1/2} \right) \right]^2
+ \caro^{-2/3} \right) \norm{\spd}^2_{C^2}
\\
&\lesssim
\caro^{-2/3} \log^2\left( \caro \right) \norm{\spd}^2_{C^2}.
\end{align*}
\end{proof}

\subsection{Proof of Proposition \ref{prop_span}}
Let $U \in \mathbb{R}^{L \times L}$ be an orthogonal matrix such that
\begin{align*}
\widetilde \newm_k = \sum_{k'=0}^{L-1} U[k,k'] \newm_{k'},
\qquad k=0, \ldots, L-1.
\end{align*}
Using that $U^\transp U = I$, we compute
\begin{align*}
&\frac{1}{L}\sum_{k=0}^{L-1} \hspd_{\widetilde \newm_k}(\xi)
=
\frac{1}{L}\sum_{k=0}^{L-1} \abs{\sum_{q \in \Zdst} \widetilde \newm_k[q] \popr[q]
\e{\ip{q}{\xi}}}^2
=
\frac{1}{L}\sum_{k=0}^{L-1}
\sum_{q,q' \in \Zdst} \widetilde \newm_k[q] \widetilde \newm_k[q']
\popr[q] \popr[q'] \e{\ip{q-q'}{\xi}}
\\
&\qquad
=
\frac{1}{L}\sum_{k=0}^{L-1}
\sum_{k'=0}^{L-1}
\sum_{k''=0}^{L-1}
\sum_{q,q' \in \Zdst}
U[k,k'] U[k,k'']
 \newm_{k'}[q] \newm_{k''}[q']
\popr[q] \popr[q'] \e{\ip{q-q'}{\xi}}
\\
&\qquad
=
\frac{1}{L} \sum_{k'=0}^{L-1}
\sum_{k''=0}^{L-1}
\sum_{q,q' \in \Zdst}
\left(
\sum_{k=0}^{L-1}
{U^\transp[k'',k]} U[k,k'] \right)
 \newm_{k'}[q] \newm_{k''}[q']
\popr[q] \popr[q'] \e{\ip{q-q'}{\xi}}
\\
&\qquad
=
\frac{1}{L} \sum_{k'=0}^{L-1}
\sum_{q,q' \in \Zdst}
 \newm_{k'}[q] \newm_{k'}[q']
\popr[q] \popr[q'] \e{\ip{q-q'}{\xi}}
=
\frac{1}{L} \sum_{k'=0}^{L-1}
\abs{
\sum_{q\in \Zdst}
 \newm_{k'}[q] \popr[q] \e{\ip{q}{\xi}}}^2
\\
&\qquad
=\frac{1}{L}\sum_{k=0}^{L-1} \hspd_{\newm_k}(\xi),
\end{align*}
as claimed. \qed

\bibliographystyle{siamplain}
\bibliography{randtap}

\begin{thebibliography}{10}

\bibitem{abgrro16}
{\sc L.~D. Abreu, K.~Gr\"ochenig, and J.~L. Romero}, {\em On accumulated
  spectrograms}, Trans. Amer. Math. Soc., 368 (2016), pp.~3629--3649,
  \url{https://doi.org/10.1090/tran/6517}.

\bibitem{abpero17}
{\sc L.~D. Abreu, J.~M. Pereira, and J.~L. Romero}, {\em Sharp rates of
  convergence for accumulated spectrograms}, Inverse Problems, 33 (2017),
  pp.~115008, 12, \url{https://doi.org/10.1088/1361-6420/aa8d79}.

\bibitem{abro17}
{\sc L.~D. Abreu and J.~L. Romero}, {\em M{SE} estimates for multitaper
  spectral estimation and off-grid compressive sensing}, IEEE Trans. Inform.
  Theory, 63 (2017), pp.~7770--7776,
  \url{https://doi.org/10.1109/TIT.2017.2718963}.

\bibitem{ansi17}
{\sc J.~And{\'e}n and A.~Singer}, {\em Factor analysis for spectral
  estimation}, in Proc. SampTA, IEEE, 2017, pp.~169--173.

\bibitem{ansi18}
{\sc J.~And\'en and A.~Singer}, {\em Structural variability from noisy
  tomographic projections}, SIAM J. Imaging Sci., 11 (2018), pp.~1441--1492,
  \url{https://doi.org/10.1137/17M1153509}.

\bibitem{bagr2009}
{\sc W.~T. Baxter, R.~A. Grassucci, H.~Gao, and J.~Frank}, {\em Determination
  of signal-to-noise ratios and spectral {SNRs} in cryo-{EM} low-dose imaging
  of molecules}, J. Struct. Biol., 166 (2009), pp.~126--132,
  \url{https://doi.org/10.1016/j.jsb.2009.02.012}.

\bibitem{be83}
{\sc J.~J. Benedetto}, {\em Harmonic analysis and spectral estimation}, J.
  Math. Anal. Appl., 91 (1983), pp.~444--509,
  \url{https://doi.org/10.1016/0022-247X(83)90164-6}.

\bibitem{MR681465}
{\sc R.~Y. Bentkus and R.~A. Rudzkis}, {\em On the distribution of some
  statistical estimates of a spectral density}, Teor. Veroyatnost. i Primenen.,
  27 (1982), pp.~739--756.

\bibitem{bhzhsi16}
{\sc T.~Bhamre, T.~Zhang, and A.~Singer}, {\em Denoising and covariance
  estimation of single particle cryo-{EM} images}, J. Struct. Biol., 195
  (2016), pp.~72--81.

\bibitem{MR847527}
{\sc O.~Brander and B.~DeFacio}, {\em A generalisation of {S}lepian's solution
  for the singular value decomposition of filtered {F}ourier transforms},
  Inverse Problems, 2 (1986), pp.~L9--L14,
  \url{http://stacks.iop.org/0266-5611/2/L9}.

\bibitem{brda91}
{\sc P.~J. Brockwell and R.~A. Davis}, {\em Time Series: Theory and Methods},
  Springer-Verlag New York, 2nd~ed., 1991.

\bibitem{bronez1988spectral}
{\sc T.~P. Bronez}, {\em Spectral estimation of irregularly sampled
  multidimensional processes by generalized prolate spheroidal sequences}, IEEE
  Trans. Acoust., Speech, Signal Process., 36 (1988), pp.~1862--1873.

\bibitem{MR3055254}
{\sc T.~T. Cai, Z.~Ren, and H.~H. Zhou}, {\em Optimal rates of convergence for
  estimating {T}oeplitz covariance matrices}, Probab. Theory Related Fields,
  156 (2013), pp.~101--143, \url{https://doi.org/10.1007/s00440-012-0422-7}.

\bibitem{coel14}
{\sc D.~Cohen and Y.~C. Eldar}, {\em Sub-{N}yquist sampling for power spectrum
  sensing in cognitive radios: a unified approach}, IEEE Trans. Signal
  Process., 62 (2014), pp.~3897--3910,
  \url{https://doi.org/10.1109/TSP.2014.2331613}.

\bibitem{dahlen2008spectral}
{\sc F.~Dahlen and F.~J. Simons}, {\em Spectral estimation on a sphere in
  geophysics and cosmology}, Geophys. J. Int., 174 (2008), pp.~774--807.

\bibitem{fr06}
{\sc J.~Frank}, {\em Three-dimensional electron microscopy of macromolecular
  assemblies}, Academic Press, 2006.

\bibitem{golub-vanloan}
{\sc G.~H. Golub and C.~F. Van~Loan}, {\em Matrix computations}, Johns Hopkins
  University Press, 3rd~ed., 1996.

\bibitem{refs2}
{\sc K.~M. Gorski}, {\em On determining the spectrum of primordial
  inhomogeneity from the {COBE DMR} sky maps: Method}, Astrophys. J., 430
  (1994), pp.~L85--L88.

\bibitem{gr81}
{\sc F.~A. Gr\"unbaum}, {\em Eigenvectors of a {T}oeplitz matrix: discrete
  version of the prolate spheroidal wave functions}, SIAM J. Algebraic Discrete
  Methods, 2 (1981), pp.~136--141, \url{https://doi.org/10.1137/0602017}.

\bibitem{MR627121}
{\sc F.~A. Gr\"{u}nbaum}, {\em Finite convolution integral operators commuting
  with differential operators: some counterexamples}, Numer. Funct. Anal.
  Optim., 3 (1981), pp.~185--199,
  \url{https://doi.org/10.1080/01630568108816086}.

\bibitem{ha17}
{\sc A.~Hanssen}, {\em Multidimensional multitaper spectral estimation}, Signal
  Process., 58 (1997), pp.~327--332.

\bibitem{hasi12}
{\sc C.~Harig and F.~J. Simons}, {\em Mapping {G}reenland's mass loss in space
  and time}, Proc. Natl. Acad. Sci. U.S.A., 109 (2012), pp.~19934--19937.

\bibitem{ha05}
{\sc S.~Haykin}, {\em Cognitive radio: brain-empowered wireless
  communications}, IEEE J. Sel. Areas Commun., 23 (2005), pp.~201--220.

\bibitem{hola12}
{\sc J.~A. Hogan and J.~D. Lakey}, {\em Duration and bandwidth limiting},
  Applied and Numerical Harmonic Analysis, Birkh\"auser/Springer, New York,
  2012, \url{https://doi.org/10.1007/978-0-8176-8307-8}.

\bibitem{hola17}
{\sc J.~A. Hogan and J.~D. Lakey}, {\em On the numerical evaluation of bandpass
  prolates {II}}, J. Fourier Anal. Appl., 23 (2017), pp.~125--140,
  \url{https://doi.org/10.1007/s00041-016-9465-y}.

\bibitem{refs1}
{\sc C.~Hwang}, {\em Orthogonal functions over the oceans and applications to
  the determination of orbit error, geoid and sea surface topography from
  satellite altimetry}, PhD thesis, The Ohio State University, 1991.

\bibitem{kabanava2017masked}
{\sc M.~Kabanava and H.~Rauhut}, {\em Masked {T}oeplitz covariance estimation},
  ArXiv preprint:1709.09377,  (2017).

\bibitem{karnikbandlimited}
{\sc S.~Karnik, J.~Romberg, and M.~A. Davenport}, {\em Bandlimited signal
  reconstruction from nonuniform samples}, Proc. Work. on Signal Processing
  with Adaptive Sparse Structured Representations (SPARS),  (2019).

\bibitem{kash18}
{\sc R.~Katz and Y.~Shkolnisky}, {\em Sampling and approximation of bandlimited
  volumetric data}, Appl. Comput. Harmon. Anal.,  (2018),
  \url{https://doi.org/10.1016/j.acha.2018.11.003}.

\bibitem{laho12}
{\sc J.~D. Lakey and J.~A. Hogan}, {\em On the numerical computation of certain
  eigenfunctions of time and multiband limiting}, Numer. Funct. Anal. Optim.,
  33 (2012), pp.~1095--1111,
  \url{https://doi.org/10.1080/01630563.2012.682133}.

\bibitem{lash16}
{\sc B.~Landa and Y.~Shkolnisky}, {\em Approximation scheme for essentially
  bandlimited and space-concentrated functions on a disk}, Appl. Comput.
  Harmon. Anal., 43 (2017), pp.~381--403,
  \url{https://doi.org/10.1016/j.acha.2016.01.006}.

\bibitem{liro08}
{\sc K.~S. Lii and M.~Rosenblatt}, {\em Prolate spheroidal spectral estimates},
  Statist. Probab. Lett., 78 (2008), pp.~1339--1348,
  \url{https://doi.org/10.1016/j.spl.2008.05.022}.

\bibitem{mahu17}
{\sc R.~Matthysen and D.~Huybrechs}, {\em Function approximation on arbitrary
  domains using {F}ourier extension frames}, SIAM J. Numer. Anal., 56 (2018),
  pp.~1360--1385, \url{https://doi.org/10.1137/17M1134809}.

\bibitem{ni01}
{\sc M.~Nielsen}, {\em On the construction and frequency localization of finite
  orthogonal quadrature filters}, J. Approx. Theory, 108 (2001), pp.~36--52,
  \url{https://doi.org/10.1006/jath.2000.3514}.

\bibitem{MR745089}
{\sc B.~N. Parlett and W.~D. Wu}, {\em Eigenvector matrices of symmetric
  tridiagonals}, Numer. Math., 44 (1984), pp.~103--110,
  \url{https://doi.org/10.1007/BF01389758}.

\bibitem{pewa93}
{\sc D.~B. Percival and A.~T. Walden}, {\em Spectral Analysis for Physical
  Applications}, Cambridge University Press, 1993.

\bibitem{plattner2014potential}
{\sc A.~Plattner and F.~J. Simons}, {\em Potential-field estimation using
  scalar and vector {S}lepian functions at satellite altitude}, Handbook of
  Geomathematics,  (2014), pp.~2003--2055.

\bibitem{plsi14}
{\sc A.~Plattner and F.~J. Simons}, {\em Spatiospectral concentration of vector
  fields on a sphere}, Appl. Comput. Harmon. Anal., 36 (2014), pp.~1--22,
  \url{https://doi.org/10.1016/j.acha.2012.12.001}.

\bibitem{risi95}
{\sc K.~S. Riedel and A.~Sidorenko}, {\em Minimum bias multiple taper spectral
  estimation}, IEEE Trans. Sig. Process., 43 (1995), pp.~188--195.

\bibitem{MR3396212}
{\sc Y.~Saad}, {\em Numerical methods for large eigenvalue problems}, SIAM,
  Philadelphia, PA, 2011, \url{https://doi.org/10.1137/1.9781611970739.ch1}.

\bibitem{relion}
{\sc S.~Scheres}, {\em {RELION}: {I}mplementation of a {B}ayesian approach to
  cryo-{EM} structure determination}, J. Struct. Biol., 180 (2012),
  pp.~519--530, \url{https://doi.org/10.1016/j.jsb.2012.09.006}.

\bibitem{MR3325822}
{\sc S.~Schmutzhard, T.~Hrycak, and H.~G. Feichtinger}, {\em A numerical study
  of the {L}egendre-{G}alerkin method for the evaluation of the prolate
  spheroidal wave functions}, Numer. Algorithms, 68 (2015), pp.~691--710,
  \url{https://doi.org/10.1007/s11075-014-9867-3}.

\bibitem{simons2006spherical}
{\sc F.~J. Simons and F.~Dahlen}, {\em Spherical slepian functions and the
  polar gap in geodesy}, Geophys. J. Int., 166 (2006), pp.~1039--1061.

\bibitem{sidawi06}
{\sc F.~J. Simons, F.~A. Dahlen, and M.~A. Wieczorek}, {\em Spatiospectral
  concentration on a sphere}, SIAM Rev., 48 (2006), pp.~504--536,
  \url{https://doi.org/10.1137/S0036144504445765}.

\bibitem{simons2003spatiospectral}
{\sc F.~J. Simons, R.~D. van~der Hilst, and M.~T. Zuber}, {\em Spatiospectral
  localization of isostatic coherence anisotropy in australia and its relation
  to seismic anisotropy: Implications for lithospheric deformation}, J.
  Geophys. Res. B, 108 (2003).

\bibitem{siwado11}
{\sc F.~J. Simons and D.~V. Wang}, {\em Spatiospectral concentration in the
  {C}artesian plane}, GEM Int. J. Geomath., 2 (2011), pp.~1--36,
  \url{https://doi.org/10.1007/s13137-011-0016-z}.

\bibitem{simons2000isostatic}
{\sc F.~J. Simons, M.~T. Zuber, and J.~Korenaga}, {\em Isostatic response of
  the {A}ustralian lithosphere: Estimation of effective elastic thickness and
  anisotropy using multitaper spectral analysis}, J. Geophys. Res. B, 105
  (2000), pp.~19163--19184.

\bibitem{sl64}
{\sc D.~Slepian}, {\em Prolate spheroidal wave functions, {F}ourier analysis
  and uncertainty -- {IV}: {E}xtensions to many dimensions; generalized prolate
  spheroidal functions}, Bell Syst. Tech. J., 43 (1964), pp.~3009--3057,
  \url{https://doi.org/10.1002/j.1538-7305.1964.tb01037.x}.

\bibitem{sl83}
{\sc D.~Slepian}, {\em Some comments on {F}ourier analysis, uncertainty and
  modeling}, SIAM Rev., 25 (1983), pp.~379--393,
  \url{https://doi.org/10.1137/1025078}.

\bibitem{slpo61}
{\sc D.~Slepian and H.~O. Pollak}, {\em Prolate spheroidal wave functions,
  {F}ourier analysis and uncertainty -- {I}}, Bell Syst. Tech. J., 40 (1961),
  pp.~43--63, \url{https://doi.org/10.1002/j.1538-7305.1961.tb03976.x}.

\bibitem{refs3}
{\sc D.~Slobbe, F.~Simons, and R.~Klees}, {\em The spherical {S}lepian basis as
  a means to obtain spectral consistency between mean sea level and the geoid},
  Journal of Geodesy, 86 (2012), pp.~609--628.

\bibitem{st73}
{\sc G.~Stewart}, {\em Error and perturbation bounds for subspaces associated
  with certain eigenvalue problems}, SIAM Rev., 15 (1973), pp.~727--764,
  \url{https://doi.org/10.1137/1015095}.

\bibitem{th82}
{\sc D.~J. Thomson}, {\em Spectrum estimation and harmonic analysis}, Proc.
  IEEE, 70 (1982), pp.~1055--1096.

\bibitem{vura2013}
{\sc M.~Vulovi\'{c}, R.~B. Ravelli, L.~J. van Vliet, A.~J. Koster,
  I.~Lazi\'{c}, U.~L\"{u}cken, H.~Rullg\r{a}rd, O.~\"{O}ktem, and B.~Rieger},
  {\em Image formation modeling in cryo-electron microscopy}, J. Struct. Biol.,
  183 (2013), pp.~19--32, \url{https://doi.org/10.1016/j.jsb.2013.05.008}.

\bibitem{rib80s}
{\sc W.~Wong, X.~Bai, et~al.}, {\em Cryo-{EM} structure of the {P}lasmodium
  falciparum {80S} ribosome bound to the anti-protozoan drug emetine}, Elife, 3
  (2014), p.~e03080.

\bibitem{xidi18}
{\sc Y.~Xiang, J.~Ding, and V.~Tarokh}, {\em Estimation of the evolutionary
  spectra with application to stationarity test}, {IEEE} Trans. Signal
  Process., 67 (2019), pp.~1353--1365.

\bibitem{zhsi13}
{\sc Z.~Zhao and A.~Singer}, {\em Fourier--{B}essel rotational invariant
  eigenimages}, J. Opt. Soc. Am. A, 30 (2013), pp.~871--877.

\bibitem{zhu2018eigenvalue}
{\sc Z.~Zhu, S.~Karnik, M.~A. Davenport, J.~Romberg, and M.~B. Wakin}, {\em The
  eigenvalue distribution of discrete periodic time-frequency limiting
  operators}, IEEE Signal Process. Lett., 25 (2018), pp.~95--99.

\bibitem{zhu2017roast}
{\sc Z.~Zhu, S.~Karnik, M.~B. Wakin, M.~A. Davenport, and J.~Romberg}, {\em
  {ROAST}: Rapid orthogonal approximate {S}lepian transform}, IEEE Trans.
  Signal Process., 66 (2018), pp.~5887--5901.

\end{thebibliography}

\end{document}